\documentclass[11pt,letterpaper,reqno]{amsart}

\usepackage{mathtools}
\usepackage{amsthm,amssymb}
\usepackage{amsmath}
\usepackage[usenames,dvipsnames]{xcolor}
\usepackage{dsfont}
\usepackage{bm}

\usepackage[shortlabels]{enumitem}
\usepackage{hyperref}
\hypersetup{colorlinks = true, linkcolor = blue, citecolor = blue, urlcolor = blue}
\usepackage{cite}

\allowdisplaybreaks

\usepackage[margin=1in]{geometry}
\usepackage[parfill]{parskip}

\usepackage{cleveref}

\usepackage{amsaddr}

\newtheorem{theorem}{Theorem}
\newtheorem*{theorem*}{Theorem}

\newtheorem{remark}{Remark}
\newtheorem{corollary}[theorem]{Corollary}
\newtheorem{lemma}[theorem]{Lemma}

\newtheorem{definition}[theorem]{Definition}
\newtheorem{claim}[theorem]{Claim}
\newtheorem*{claim*}{Claim}

\newcommand{\diff}{\mathrm{d}}
\newcommand{\E}{\mathbb{E}}
\newcommand{\R}{\mathbb{R}}

\newcommand{\N}{\mathbb{N}}
\newcommand{\Z}{\mathbb{Z}}
\newcommand{\ind}[1]{\mathds{1}_{\{#1\}}}
\newcommand{\absolute}[1]{\left\lvert #1 \right\rvert}
\newcommand{\eulerE}{\textup{e}}
\newcommand{\dtv}[2]{\absolute{#1 - #2}_{\mathrm{tv}}}
\newcommand{\projDtv}[3]{\absolute{#1 - #2}_{#3}}
\DeclareMathOperator*{\esssup}{ess\,sup}
\newcommand{\symdiff}{\oplus}
\newcommand{\extended}[1]{\overline{#1}}
\newcommand{\dimension}{d}
\newcommand{\norm}[2]{\lVert #1 \rVert_{#2}}
\newcommand{\paths}{\mathfrak{P}}

\newcommand{\vertices}{V}
\newcommand{\edges}{E}
\newcommand{\innerVertices}{\vertices^{\textup{in}}}
\newcommand{\outerVertices}{\vertices^{\textup{out}}}
\newcommand{\graph}{G}
\newcommand{\neighborhood}{\Gamma}
\newcommand{\tuple}[1]{\pmb{#1}}
\newcommand{\disagreement}{\Delta}

\newcommand{\range}{R}

\newcommand{\groundspace}{\mathcal{X}}
\newcommand{\borel}{\mathcal{B}}
\newcommand{\boundedBorel}{\borel_{\text{b}}}
\newcommand{\volume}{\nu}
\newcommand{\countingMeasures}{\mathcal{N}}
\newcommand{\finiteCountingMeasures}{\countingMeasures_{\text{f}}}
\newcommand{\feasibleCountingMeasures}{\countingMeasures_{\hamiltonian}}
\newcommand{\countingAlgebra}{\mathcal{R}}
\newcommand{\finiteCountingAlgebra}{\countingAlgebra_{\text{f}}}
\newcommand{\support}{S}
\newcommand{\pointcount}{N}
\newcommand{\subregion}{\Lambda}
\newcommand{\probMeasures}{\mathcal{P}}
\newcommand{\hamiltonian}{H}
\newcommand{\potential}{\phi}
\newcommand{\activity}{\lambda}
\newcommand{\partitionFunction}{\Xi}
\newcommand{\gibbs}{\mu}
\newcommand{\normalBoundary}{\mathcal{C}}
\newcommand{\gibbsMeasures}{\mathcal{G}}
\newcommand{\stability}{S}
\newcommand{\localStability}{L}
\newcommand{\tempered}{C}
\newcommand{\weakTempered}{\hat{C}}
\newcommand{\influence}{W}

\newcommand{\jumpKernel}{K}

\newcommand{\filtration}{\mathcal{F}}

\renewcommand{\Pr}{\mathbb{P}}
\newcommand{\dist}{d}

\newcommand{\size}[1]{\left\lvert #1 \right\rvert}

\usepackage{soul}
\definecolor{foo}{HTML}{f7b500}
\sethlcolor{foo}

\title{Uniqueness of locally stable Gibbs point processes via spatial birth-death dynamics}
\author{Samuel Baguley, Andreas Göbel and Marcus Pappik}
\address{Hasso Plattner Institute, University of Potsdam, Germany}

\email{\{samuel.baguley, andreas.goebel, marcus.pappik\}@hpi.de}

\begin{document}

\begin{abstract}
    We prove that for every locally stable and tempered pair potential $\potential$ with bounded range, there exists a unique infinite-volume Gibbs point process on $\R^{\dimension}$ for every activity $\activity < (\eulerE^{\localStability} \weakTempered_{\potential})^{-1}$, where $\localStability$ is the local stability constant and $\weakTempered_{\potential} \coloneqq \sup_{x \in \R^{\dimension}} \int_{\R^{\dimension}} 1 - \eulerE^{-\absolute{\potential(x, y)}} \diff y$ is the (weak) temperedness constant. 
    Our result extends the uniqueness regime that is given by the classical Ruelle--Penrose bound by a factor of at least $\eulerE$, where the improvements becomes larger as the negative parts of the potential become more prominent (i.e., for attractive interactions at low temperature).
    Our technique is based on the approach of Dyer et al. (Rand. Struct. \& Alg. ’04): we show that for any bounded region and any boundary condition, we can construct a Markov process (in our case spatial birth-death dynamics) that converges rapidly to the finite-volume Gibbs point process while effects of the boundary condition propagate sufficiently slowly. 
    As a result, we obtain a spatial mixing property that implies uniqueness of the infinite-volume Gibbs measure.
\end{abstract}

\maketitle

\begin{center}
\begin{minipage}{0.9\textwidth} 
\end{minipage}
\end{center}

\begin{center}
\begin{minipage}{0.9\textwidth} 
  \textbf{Funding:} Andreas Göbel was funded by the project PAGES (project No. 467516565) of the German Research Foundation (DFG). Marcus Pappik was funded by the HPI Research School on Data Science and Engineering.
\end{minipage}
\end{center}

\section{Introduction}
In statistical physics, point processes play a central role as a mathematical model for gasses of interacting particles: the gas is modeled as a random point configuration in a ground space $\groundspace$, where every point in a configuration represents the location of a particle.
A corner stone of this approach for modelling gasses are Gibbs point processes, which historically arise from describing the distribution of particles as the solution to a variational principle that aims at minimizing the mean free excess energy of the system (see for example \cite[Section 1.3]{dereudre2019introduction}).
While these ideas date back all the way to Gibbs \cite{gibbs1902elementary} (and even further to the work of Boltzmann and Maxwell), the theory of Gibbs point processes was substantially developed further throughout the 60s and 70s by the work of Dobrushin \cite{dobruschin1968description}, Lanford and Ruelle \cite{lanford1969observables}, Preston \cite{preston1976random}, Georgii \cite{georgii1976canonical} and many others, and it remains an active area of research till the present days \cite{houdebert2022explicit,michelen2020connective,he2024analyticity,michelen2020analyticity,dereudre2020existence,jansen2019cluster,betsch2023uniqueness} (see also \cite{ruelle1999statistical} for a classical and  \cite{dereudre2019introduction,jansen2018gibbsian} for more modern introductions to the topic).

Given a finite-volume region $\subregion \subset \groundspace$ of the ground space $\groundspace$ (for simplicity, think $\groundspace = \R^{\dimension}$), a Gibbs point process on $\subregion$ is described by a probability density with respect to a Poisson point process on $\subregion$ of intensity 1.
This probability density is specified by an activity $\activity \in \R_{\ge 0}$ and an energy function $\hamiltonian$, which assigns a value from $(-\infty, \infty]$ to finite point sets $\eta \subset \groundspace$.
Given $\activity$ and $\hamiltonian$, every finite point set $\eta \subset \subregion$ has a density proportional to $\activity^{\size{\eta}} \eulerE^{-H(\eta)}$.

Opposed to this finite-volume setting, statistical physics is usually interested in large particle systems. 
Such systems are more accurately modeled by considering point processes on regions of infinite volume, while the configuration space is extended to all locally finite point sets (i.e., we only require finitely many points in every bounded subregion).
In this case, the definition via probability densities often fails as the density cannot be normalized properly and configurations typically have infinite energy.
Instead, given an activity $\activity$ and energy function $\hamiltonian$, an infinite-volume Gibbs point process is defined by the \emph{Dobrushin--Lanford--Ruelle formalism} \cite{dobruschin1968description,lanford1969observables} (DLR), which imposes a set of consistency requirements: any finite-volume projection of an infinite-volume Gibbs point process should be consistent with its finite-volume definition discussed above (see \Cref{def:dlr}).
Due to the non-constructive nature of this approach, two main questions arise.
Given an activity $\activity$ and an energy function $\hamiltonian$:
\begin{itemize}
    \item \emph{Existence:} is there at least one point process on $\groundspace$ that satisfies the DLR formalism?
    \item \emph{Uniqueness:} is there at most one point process on $\groundspace$ that satisfies the DLR formalism?\footnote{Note that usually in mathematics, uniqueness means that there is exactly one object that satisfies the requirements. However, for Gibbs point processes, this task is usually split into proving that there is a most one and at least one Gibbs measure. Hence it will be more convenient to refer to these two problems as existence and uniqueness respectively.}
\end{itemize}

While the existence of an infinite-volume Gibbs measures can be proven under fairly general assumptions, which are usually stated in terms of the energy function $\hamiltonian$ alone, proving uniqueness is notoriously more difficult \cite{dereudre2019introduction,jansen2018gibbsian,dereudre2020existence}.
In particular, conditions for uniqueness often involve that the activity $\activity$ is sufficiently small, depending on various properties of the energy function (see for example \cite{helmuth2022correlation,michelen2020analyticity,michelen2020connective}).

On top of being a genuinely intriguing mathematical question, uniqueness of infinite-volume Gibbs measures in the sense of DLR gained a lot of attention in mathematical physics as a probabilistic notion of \emph{absence of phase transitions} \cite{jansen2018gibbsian,michelen2020analyticity,michelen2020connective,houdebert2022explicit,betsch2023uniqueness}.
As such, it poses an alternative to (much older) analytic notions, such as analyticity of the infinite-volume pressure (see \cite[Section 4]{jansen2018gibbsian} for a detailed introduction), which are usually proven by studying the convergence of series expansions \cite{penrose1963convergence,procacci2017convergence,jansen2022cluster} or the contraction integral identities \cite{michelen2020analyticity,michelen2020connective,he2024analyticity}.
While in many cases such analytic proof techniques can be modified to yield uniqueness in the DLR sense (see for example \cite{jansen2019cluster} and \cite{michelen2020analyticity}) and although both notions of absence of phase transitions often hold in similar activity regimes, to the best of our knowledge no general rigorous connection has been proved.
We defer a more detailed discussion to \Cref{sec:discussion} and proceed to sketching our main result.

\subsection{Main result and proof technique}
In this paper, we focus on giving an activity regime for uniqueness of the infinite-volume Gibbs point process for a class of energy functions that result from pair interactions,
that is, the case that there is some pair potential $\potential: \groundspace^2 \to (-\infty, \infty]$ such that $\hamiltonian(\eta) = \sum_{x, y \in \binom{\eta}{2}} \potential(x, y)$, where we sum over unordered pairs of distinct points in $\eta$.
Moreover, we assume the following:
\begin{itemize}
    \item \emph{Weak temperedness:} We call $\potential$ weakly tempered if
    \[
        \weakTempered_{\potential} \coloneqq \sup_{x \in \groundspace} \int_{\groundspace} 1 - \eulerE^{- \absolute{\potential(x, y)}} \diff y < \infty ,
    \]
    and we refer to $\weakTempered_{\potential}$ as the weak temperedness constant. This condition should be thought of as assuming that the total influence of every point on the rest of the space is bounded.\footnote{While this definition has appeared before in the literature (e.g., \cite{procacci2017convergence}), we are not aware of it having a name. We call it weak temperedness due to its resemblance to the temperedness condition $\tempered_{\potential} \coloneqq \sup_{x \in \groundspace} \int_{\groundspace} \lvert 1 - \eulerE^{- \potential(x, y)}\rvert \diff y < \infty$ and the fact that $\weakTempered_{\potential} \le \tempered_{\potential}$.}
    \item \emph{Local stability:} We call $\potential$ locally stable with constant $\localStability \in \R_{\ge 0}$ if for every finite $\eta \subset \groundspace$ with $\hamiltonian(\eta) < \infty$ and every $x \in \groundspace$ it holds that
    \[
        \hamiltonian(\eta \cup x) - \hamiltonian(\eta) \ge - \localStability.
    \]
    This condition should be thought of as assuming that for any realizable configuration of points $\eta$, adding a point $x$ decreases the energy of the configuration by at most $\localStability$.
    \item \emph{Bounded range:} We say that the has bounded range if there exists some $\range \in \R_{\ge 0}$ such that $\potential(x, y) = 0$ for every pair of points $x, y \in \groundspace$ of distance at least $\range$.
\end{itemize}

Our main result is the following uniqueness statement for Gibbs point processes in Euclidean space $\groundspace = \R^{\dimension}$.
\begin{theorem*}[{Main}]
    For every bounded-range pair potential $\potential$ on $\R^{\dimension}$ that is locally stable with constant $\localStability$ and weakly tempered with constant $\weakTempered_{\potential}$, and every activity 
$\activity < (\eulerE^{\localStability} \weakTempered_{\potential})^{-1}$,
there exists a unique Gibbs point process on $\R^{\dimension}$ in the sense of the Dobrushin--Lanford--Ruelle formalism.
\end{theorem*}

The formal version of the above theorem is given in \Cref{thm:uniqueness} and \Cref{cor:uniqueness}.
We defer a detailed comparison between our activity bound and the existing literature to \Cref{sec:discussion}, where we compare it to various results that prove absence of a phase transition (in the sense of DLR or analytically) in a similar setting.
However, we remark that, for bounded-range locally stable potentials, our bound provides an improvement by a factor of $\eulerE \tempered_{\potential}/\weakTempered_{\potential} \ge \eulerE$ over the classical Penrose--Ruelle bound \cite{ruelle1963correlation,penrose1963convergence}, which implies uniqueness via contraction of the Kirkwood--Salsburg equations, and which constituted the best known bound in this setting for more than 50 years.
We proceed with discussing how we obtain our result.

In the setting of our main theorem, existence of the infinite-volume Gibbs point process follows from known results (see \Cref{thm:existence}), and our main contribution is to prove that the infinite-volume Gibbs measure is unique.
Our argument is purely probabilistic and inspired by the work of Dyer, Sinclair, Vigoda and Weitz \cite{dyer2004mixing} for discrete hard-core gasses on a lattice.
The main idea is to construct a Markov process (in our case a spatial birth-death process as introduced by Preston \cite{preston1975spatial}) for every finite-volume region $\subregion$ such that the process is reversible (hence stationary) with respect to the finite-volume Gibbs measure on $\subregion$.
We then establish two properties for this family of Markov processes.

Firstly, we prove a \emph{rapid mixing} property, meaning that these processes converge rapidly to their stationary distribution in total variation distance.
We show this rapid mixing result for all activities $\activity < (\eulerE^{\localStability} \weakTempered_{\potential})^{-1}$ by arguing that the local stability and weak temperedness assumptions imply contraction of a suitably constructed coupling.
This step is fairly robust with respect to the underlying space (it holds on arbitrary complete separable metric spaces) and does not rely on the bounded range of the potential.

Secondly, we prove \emph{slow disagreement percolation}, which means that, for each of these Markov processes, local disagreements between initial configurations spread slowly.
For this step, we use the local stability along with the bounded range assumption and the fact that Euclidean space ``grows slowly''.

Combining both properties above shows that, under the conditions of the main theorem, finite-volume Gibbs point processes satisfy a \emph{spatial mixing} property: different parts of a random point configuration become asymptotically independent as the distance increases.
We use this observation to show that any pair of point processes satisfying the DLR formalism must have identical finite-volume projections. 
By a standard result \cite[Corollary 9.1.IX]{daley2007introduction}, every point process is uniquely characterized by its finite-volume projections, and hence there is at most one infinite-volume Gibbs measure.  

The proof technique we sketched above was previously applied in the continuous setting by Helmuth, Perkins and Petti \cite{helmuth2022correlation}, who used a discrete-time Markov chain to study the special case of uniqueness of the infinite-volume hard-sphere model.
Our work may be seen as a natural extension to more general potentials.
Moreover, spatial birth-death processes have been used before to study different aspects of Gibbs point processes \cite{bertini2002spectral,schuhmacher2014gibbs,kondratiev2005glauber,kondratiev2013spectral,fernandez2016stability}, sometimes under different name such as \emph{continuum Glauber dynamics}. 
In particular the generator and its spectral gap have been studied in both the finite-volume \cite{bertini2002spectral} and infinite-volume setting \cite{kondratiev2005glauber,kondratiev2013spectral}, and the results have been used to prove absence of a phase transition for various classes of models \cite{kondratiev2013spectral}.
However none of these results imply or are implied by ours. 
Moreover, as our result is based on a simple coupling argument, most of the theory on spatial birth-death processes that we require holds in the setting of jump processes on general state spaces as presented in the classical texts of Feller \cite{feller1991introduction} and Blumenthal and Getoor \cite{BG}.
This allows us to keep the presentation of our result and our proofs surprisingly elementary.

\subsection{Outline of the paper}
The rest of the paper is structured as follows. 
In \Cref{sec:gpp} we formally introduce finite- and infinite-volume Gibbs point processes along with several useful properties.
In particular, we present a sufficient condition for uniqueness of infinite-volume Gibbs measures, which essentially constitutes a spatial mixing property.
As we aim to keep most of our discussions as general as possible, we introduce point processes in the setting of complete separable metric spaces using the notion of random counting measures.
However, switching between this perspective and the intuition of random point sets is straightforward.
In \Cref{sec:jump_process} we give a short introduction to jump processes, which contain the spatial birth-death processes we aim to use. 
We further prove various lemmas that will allow to bound the speed of convergence to the stationary distribution and to obtain tail bounds on certain hitting times, the latter of which will be useful for the disagreement percolation.
In \Cref{sec:jump_process_for_gpp} we then construct the family of spatial birth-death processes that we study and prove the desired rapid mixing and slow disagreement percolation.
Finally, in \Cref{sec:uniqueness}, we prove our main uniqueness result that we sketched above.
We end our paper by discussing our result and compare our bound with the existing literature in \Cref{sec:discussion}.

\section{Gibbs Point Processes} \label{sec:gpp}
\subsection{Spaces and basic notation}
Let $(\groundspace, \dist)$ be a complete separable metric space.
We write $\borel$ for the Borel $\sigma$-field on $\groundspace$ and we assume that $(\groundspace, \borel)$ is equipped with a locally finite reference measure $\volume$ (i.e., $\volume(B) < \infty$ for every bounded set $B \in \borel$).
We write $\boundedBorel$ for the set of bounded Borel sets in $\borel$, and, for $\subregion \in \borel$, we write $\borel_{\subregion}$ for the trace of $\subregion$ in $\borel$.

We denote by $\countingMeasures$ the space of locally-finite counting measures on $(\groundspace, \borel)$, where we write $\mathbf{0} \in \countingMeasures$ for the constant $0$ measure.
We equip $\countingMeasures$ with the $\sigma$-field $\countingAlgebra$, generated by the set of functions $\{\pointcount_\subregion: \countingMeasures \to \N_0 \cup \{\infty\}, \eta \mapsto \eta(\subregion) \mid \subregion \in \borel\}$.
Further, we write $\finiteCountingMeasures$ for the set of finite counting measures on $(\groundspace, \borel)$ and equip it with $\finiteCountingAlgebra$, the trace of $\finiteCountingMeasures$ in $\countingAlgebra$.
For $\eta \in \countingMeasures$ and $\subregion \in \borel$, we write $\eta_{\subregion}$ for the unique counting measure with $\eta_{\subregion}(\Delta) = \eta(\Delta \cap \subregion)$ for every $\Delta \in \borel$.
We denote by $\countingMeasures_{\subregion} \coloneqq \{\eta \in \countingMeasures \mid \eta(\groundspace \setminus \subregion) = 0\}$ the image of $\countingMeasures$ under $\eta \mapsto \eta_{\subregion}$ and write $\countingAlgebra_{\subregion}$ for the trace of $\countingMeasures_{\subregion}$ in $\countingAlgebra$.

\subsection{Integration and factorial measures}
In the above setting, it is well known that every $\eta \in \countingMeasures$ can be written as a sum of Dirac measures $\eta = \sum_{i \in I} \delta_{x_i}$, where  $(x_i)_{i \in I}$ is a sequence of points in $\groundspace$, and the index set is either $I = \{1, \dots, \eta(\groundspace)\}$ if $\eta \in \finiteCountingMeasures$ or $I = \N$ if $\eta \in \countingMeasures \setminus \finiteCountingMeasures$ (see \cite[Proposition 9.1.III]{daley2007introduction}).
As a consequence, we have for every $\eta \in \countingMeasures$ and every measurable $f: \groundspace \to \extended{\R_{\ge 0}}$ that $\int f \diff \eta = \sum_{i \in I} f(x_i)$.
This identity extends to measurable functions $f: \groundspace \to \R \cup \{\infty\}$ provided that the negative or positive part of the integral is finite, which particularly holds if $\eta \in \finiteCountingMeasures$. 

Further, given a locally-finite counting measure $\eta = \sum_{i \in I} \delta_{x_i} \in \countingMeasures$ and a positive integer $k \in \N$, we write $\eta^{(k)}$ for the $k^{\text{th}}$ factorial measure of $\eta$, which is a counting measure on the product space $(\groundspace, \borel)^{\otimes k}$ given by $\eta^{(k)} = \sum_{i_1, \dots, i_k \in I}^{\neq} \delta_{(x_{i_1}, \dots, x_{i_k})}$, where the sum is taken over $k$ tuples of distinct indices from $I$.  
Analogously to before, we have $\int_{\groundspace^k} f \diff \eta^{(k)} = \sum_{i_1, \dots, i_k \in I}^{\neq} f(x_{i_1}, \dots, x_{i_{k}})$ for every measurable $f: \groundspace^k \to \R \cup \{\infty\}$ such that the negative or positive part of the integral is finite.

\subsection{Support set and operations on counting measures}
An alternative way of writing a locally finite counting measure $\eta \in \countingMeasures$ on a complete separable metric space involves the notion of the \emph{support set}
\[
    \support_{\eta} \coloneqq \{x \in \groundspace \mid \eta(x) > 0\} ,
\]
where we write $\eta(x)$ for $\eta(\{x\})$.
It can be shown that $\support_{\eta}$ is at most countable, and, for every $x \in \support_{\eta}$, the fact that $\eta$ is locally finite implies $\eta(x) < \infty$ (see again \cite[Proposition 9.1.III]{daley2007introduction}).
We can now write $\eta$ as
\[
    \eta = \sum_{x \in \support_{\eta}} \eta(x) \cdot \delta_x .
\]
Using this, we can define the following operations for pairs of locally finite counting measures $\eta, \xi \in \countingMeasures$:
\begin{align*}
    \eta \cap \xi &\coloneqq \sum_{x \in \support_{\eta} \cap \support_{\xi}} \min(\eta(x), \xi(x)) \cdot \delta_x \\
    \eta \cup \xi &\coloneqq \sum_{x \in \support_{\eta} \cup \support_{\xi}} \max(\eta(x), \xi(x)) \cdot \delta_x \\
    \eta \setminus \xi &\coloneqq \sum_{x \in \support_{\eta}} \max(0, \eta(x) - \xi(x)) \cdot \delta_x \\
    \eta \symdiff \xi &\coloneqq (\eta \setminus \xi) + (\xi \setminus \eta) . \\
\end{align*}
Thinking of locally finite counting measures as point (multi-)sets, these operation correspond directly to intersection, union, set difference and symmetric difference. 
It is easily checked that $\eta \cup \xi = (\eta \symdiff \xi) + (\eta \cap \xi)$ for all $\eta, \xi \in \countingMeasures$, which will turn out to be a particularly useful identity.

\subsection{Point processes}
We call a probability measure $P \in \probMeasures(\countingMeasures, \countingAlgebra)$ a \emph{point process} on $\subregion \in \borel$ if $P$ is only supported on $\countingMeasures_{\subregion}$ (i.e., $P(\countingMeasures \setminus \countingMeasures_{\subregion}) = 0$).
Note that every point process $P$ on $\subregion$ can be uniquely identified with a probability measure on $(\countingMeasures_{\subregion}, \countingAlgebra_{\subregion})$. 
Moreover, for $\subregion \in \borel$ and a point process $P$ on $\groundspace$, we write $P[\subregion]$ for the pushforward of $P$ under $\eta \mapsto \eta_{\subregion}$, viewed as a map from $(\countingMeasures, \countingAlgebra)$ to itself.
That is, for all $A \in \countingAlgebra$ we have
\[
    P[\subregion](A) = \int_{\countingMeasures} \ind{\eta_{\subregion} \in A} P(\diff \eta).
\]
Note that $P[\subregion]$ is a point process on $\subregion$. 
We call $P[\subregion]$ the projection of $P$ to $\subregion$.

In fact, point processes are fully characterized by their projections to finite-volume regions.
This is implied by the following even stronger fact, which states that two point processes are equal if they induce the same joint distributions of point counts on families of bounded sets. 
To this end, recall that we defined a measurable map $\pointcount_{\subregion}: \countingMeasures \to \N_0 \cup \{\infty\}, \eta \mapsto \eta(\subregion)$ for each $\subregion \in \borel$.
\begin{lemma}[{\cite[Corollary 9.1.IX]{daley2007introduction}}] \label{fact:equality}
    Let $P, Q \in \probMeasures(\countingMeasures, \countingAlgebra)$.
    Suppose for all $\ell \in \N$, $\subregion_1, \dots, \subregion_{\ell} \in \boundedBorel$ and $m_1, \dots, m_{\ell} \in \N_0$ it holds that 
    \[
        P(\pointcount_{\subregion_1} = m_1, \dots, \pointcount_{\subregion_{\ell}} = m_{\ell})
        = Q(\pointcount_{\subregion_1} = m_1, \dots, \pointcount_{\subregion_{\ell}} = m_{\ell}).
    \]
    Then $P=Q$.
\end{lemma}

\subsection{Finite-volume Gibbs point processes }
Let $\potential: \groundspace^2 \to \R \cup \{\infty\}$ be symmetric and measurable.
We call $\potential$ a \emph{(stable) pair potential} if there is a constant $\stability \in \R_{\ge 0}$ such that for all 
$\eta \in \finiteCountingMeasures$ it holds that
\[
    \hamiltonian(\eta) \coloneqq \frac{1}{2} \int_{\groundspace^2} \potential(x, y) \eta^{(2)} (\diff x, \diff y) \ge - \stability \cdot \eta(\groundspace) .
\]
We call every such $\stability$ a \emph{stability constant} for $\potential$, and we call $H: \finiteCountingMeasures \to \R \cup \{\infty\}$ the \emph{energy function} induced by $\potential$.
Note that such an energy function satisfies the following important properties:
\begin{enumerate}
    \item It holds that $H(\mathbf{0}) = 0$ (\emph{non-degenerate}).
    \item Suppose $\eta, \xi \in \finiteCountingMeasures$ are such that $\xi \le \eta$ and $H(\xi) = \infty$; then $H(\eta) = \infty$ (\emph{hereditary}).
\end{enumerate}

Given a pair potential $\potential$ and a bounded set $\subregion \in \borel$, we define a function $\hamiltonian_{\subregion}: \countingMeasures_{\subregion} \times \countingMeasures \to \R \cup \{\infty\}$ as 
\[
    \hamiltonian_{\subregion}(\eta \mid \xi) \coloneqq \hamiltonian(\eta) + \int_{\subregion} \int_{\groundspace \setminus \subregion} \potential(x, y) \xi(\diff y) \eta(\diff x) ,
\]
where we set $\hamiltonian_{\subregion}(\eta \mid \xi) = \infty$ whenever the latter integral fails to converge absolutely. 

Next, we introduce the grand-canonical partition function, which serves as the normalizing constant of the Gibbs point process.
For $\subregion$ as above, we define the \emph{grand-canonical partition function} on $\subregion$ with boundary condition $\xi \in \countingMeasures$ and activity $\activity \in \R_{\ge 0}$ by
\[
    \partitionFunction_{\subregion \mid \xi} (\activity) \coloneqq 1 + \sum_{n = 1}^{\infty} \frac{\activity^n}{n!} \int_{\subregion^n} \exp \left( - \hamiltonian_{\subregion}\Big(\sum_{i=1}^{n} \delta_{x_i} \Big| \xi\Big) \right) \volume^{n}(\diff \tuple{x}).
\]
Obviously, it holds that $\partitionFunction_{\subregion \mid \xi} (\activity) \ge 1$. 
If further $\xi$ is such that $\partitionFunction_{\subregion \mid \xi} (\activity) < \infty$, we define the \emph{finite-volume Gibbs point process} (or finite-volume Gibbs measure) on $\subregion$ with boundary condition $\xi$ at activity $\activity$ as a probability distribution $\gibbs_{\subregion \mid \xi} \in \probMeasures(\countingMeasures, \countingAlgebra)$ with
\[
    \gibbs_{\subregion \mid \xi}(A) \coloneqq\frac{\eulerE^{\volume(\subregion)}}{\partitionFunction_{\subregion \mid \xi} (\activity)} \int_{A}  \activity^{\eta(\subregion)} \eulerE^{-\hamiltonian_{\subregion}(\eta \mid \xi)} Q_{\subregion}(\diff \eta)
\]
for $A \in \countingAlgebra$, where $Q_{\subregion}$ is the Poisson point process with intensity measure $\volume(\cdot \cap \subregion)$.
Note that $\gibbs_{\subregion \mid \xi}$ is a point process on $\subregion$ since $Q_{\subregion}$ is only supported on $\countingMeasures_{\subregion}$.

We proceed by introducing some additional terminology related to pair potentials.
As sketched in the introduction, we say a pair potential $\potential$ is \emph{weakly tempered}\footnote{An alternative condition that is frequently considered is temperedness, which requires that $\tempered_{\potential} \coloneqq \sup_{x \in \groundspace} \int_{\groundspace} |1 - \eulerE^{- \potential(x, y)}| \volume(\diff y) < \infty$. We remark that the term 'weak' does not refer to weak temperedness being a weaker condition (as $\weakTempered_{\potential} < \infty$ if and only if $\tempered_{\potential} < \infty$), but to the fact that $\weakTempered_{\potential} \le \tempered_{\potential}$, where the inequality is strict if $\potential$ is tempered and as a non-trivial negative part.} if 
\[
    \weakTempered_{\potential} \coloneqq \sup_{x \in \groundspace} \int_{\groundspace} 1 - \eulerE^{- \absolute{\potential(x, y)}} \volume(\diff y) < \infty,
\]
and we call $\weakTempered_{\potential}$ the \emph{weak temperedness constant} of $\potential$.
Moreover, we say that $\potential$ has bounded range $\range \in \R_{\ge 0}$ if for all $x, y \in \groundspace$ with $\dist(x, y) \ge 
\range$ it holds that $\potential(x, y) = 0$.
Finally, we call $\potential$ \emph{locally stable} with constant $\localStability \in \R_{\ge 0}$ if, for all $\eta \in \finiteCountingMeasures$ with $\hamiltonian(\eta) < \infty$ and all $x \in \groundspace$, it holds that
\[
    \hamiltonian(\eta + \delta_x) - \hamiltonian(\eta) \ge - \localStability .
\]

\subsection{Infinite-volume Gibbs point process}
Given an activity $\activity \in \R_{\ge 0}$ and a pair potential $\potential$, we denote by $\gibbsMeasures(\activity, \potential) \subseteq \probMeasures(\countingMeasures, \countingAlgebra)$ the set of infinite-volume Gibbs point processes that are compatible with $\activity$ and $\potential$, which we characterize via the Dobrushin-Lanford-Ruelle formalism.

\begin{definition}[DLR formalism] \label{def:dlr}
    Given an activity $\activity$ and potential $\potential$, we say a point process $\gibbs \in \probMeasures(\countingMeasures, \countingAlgebra)$ satisfies the DLR formalism if for all $\subregion \in \boundedBorel$ the following holds:
    \begin{enumerate}
        \item for $\normalBoundary_{\subregion} = \{\xi \in \countingMeasures \mid \partitionFunction_{\subregion \mid \xi}(\activity) < \infty \}$ we have $\gibbs(\normalBoundary_{\subregion}) = 1$, and
        \item for all measurable $f: \countingMeasures \to \R_{\ge 0}$ we have
        \[
            \int_{\countingMeasures} f(\eta) \gibbs(\diff \eta) = \int_{\countingMeasures} \int_{\countingMeasures_{\subregion}} f(\eta + \xi_{\groundspace \setminus \subregion}) \gibbs_{\subregion \mid \xi}(\diff \eta) \gibbs(\diff \xi) .
        \]
    \end{enumerate}
    In this case, we call $\gibbs$ an \emph{infinite-volume Gibbs point process} (or infinite-volume Gibbs measure) compatible with $\activity$ and $\potential$.
    Further, we write $\gibbsMeasures(\activity, \potential)$ for the set of such infinite-volume Gibbs point processes.
\end{definition}

As discussed earlier, we will mostly focus on proving conditions for which $\size{\gibbsMeasures(\activity, \potential)} \le 1$ (i.e., uniqueness of infinite-volume Gibbs point processes).
However, at this point, we would like to point out that, in the setting studies in this paper, it also holds that $\gibbsMeasures(\activity, \potential) \neq \emptyset$ (i.e., there also exists an infinite-volume Gibbs point process).
This is formally given by the following standard result.
\begin{theorem}[{Existence, \cite[Theorem 5.6]{jansen2018gibbsian}}] \label{thm:existence}
    Suppose $\potential$ is locally stable and has bounded range, then for all $\activity \in \R_{\ge 0}$ it holds that $\gibbsMeasures(\activity, \potential) \neq \emptyset$. 
\end{theorem}

We remark that this is not the most general setting in which existence of infinite-volume Gibbs measures can be proven, but it suffices for our purposes.
The interested reader should refer to \cite[Section 2]{dereudre2019introduction} and \cite[Section 5]{jansen2018gibbsian} for a better overview.

\subsection{Useful properties.}
We proceed by proving a list of properties of Gibbs point processes that will be useful throughout the rest of the paper.
We start by constructing the space of \emph{feasible counting measures}
\begin{align}
    \feasibleCountingMeasures \coloneqq \{\eta \in \countingMeasures \mid \text{ for all $\xi \in \finiteCountingMeasures$ with $\xi \le \eta$ it holds that $\hamiltonian(\xi) < \infty$}\} .
    \label{eq:feasible}
\end{align}
Feasible counting measures will play an important role in this work as boundary conditions.
In particular, in light of the DLR formalism in \Cref{def:dlr}, the following lemma allows us to restrict to boundary conditions from $\feasibleCountingMeasures$ when characterizing infinite-volume Gibbs measures.
\begin{lemma} \label{lemma:feasible}
    For every pair potential $\potential$ it holds that $\feasibleCountingMeasures \in \countingAlgebra$. 
    Moreover, for every activity $\activity \in \R_{\ge 0}$ and all $\gibbs \in \gibbsMeasures(\activity, \potential)$ it holds that $\gibbs(\feasibleCountingMeasures) = 1$.
\end{lemma}
\begin{proof}
    Fix some $x \in \groundspace$ and, for $n \in \N$, let $\subregion_n = \{y \in \groundspace \mid \dist(x, y) < n\}$.
    If $\eta \in \countingMeasures \setminus \feasibleCountingMeasures$ then, by definition, there is some $\xi \in \finiteCountingMeasures$ with $\xi \le \eta$ and $\hamiltonian(\xi) = \infty$.
    By the fact that $\hamiltonian$ is hereditary and $\xi$ is finite, we can find some $n \in \N$ such that $\hamiltonian(\eta_{\subregion_n}) = \infty$.
    Moreover, if such an $n \in \N$ exists, then obviously $\eta \in \countingMeasures \setminus \feasibleCountingMeasures$.
    By setting $A_n = \{ \eta \in \countingMeasures \mid \hamiltonian(\eta_{\subregion_n}) = \infty\}$ and noting that $A_n \in \countingAlgebra$, we have $\countingMeasures \setminus \feasibleCountingMeasures = \bigcup_{n \in \N} A_n  \in \countingAlgebra$.
    
    Next, suppose that $\gibbs \in \gibbsMeasures(\activity, \potential)$.
    By the union bound, it suffices to show that $\gibbs(A_n) = 0$ for every $n \in \N$.
    Observe that for all $\xi \in \countingMeasures$ and all $\eta \in A_n$ it holds that $\hamiltonian_{\subregion_n}(\eta \mid \xi) = \infty$.
    Therefore, $\gibbs_{\subregion_n \mid \xi} (A_n) = 0$ for all $\xi$ such that $\gibbs_{\subregion_n \mid \xi}$ is well-defined.
    But then, by the DLR formalism in \Cref{def:dlr}, it follows that $\gibbs(A_n) = 0$, which concludes the proof.
\end{proof}

As we will see shortly, configurations from $\feasibleCountingMeasures$ are feasible in the sense that, for every $\xi \in \feasibleCountingMeasures$, the finite-volume Gibbs measure $\gibbs_{\subregion \mid \xi}$ is always well defined (i.e., the partition function converges).
Before we prove this, it will be convenient to introduce the following notion of \emph{influence} of a configuration $\eta \in \countingMeasures$ on a single point $x \in \groundspace$ by setting
\[
    \influence(x, \eta) \coloneqq \int_{\groundspace} \potential(x, y) \eta(\diff y)
\]
whenever the integral converges absolutely and $\influence(x, \eta) \coloneqq \infty$ otherwise.
This definition of influence is particularly useful due to the following decompositions of $\hamiltonian_{\subregion}$.
\begin{lemma} \label{lemma:energy_decomposition}
    For all $\subregion \in \boundedBorel$, $\eta \in \countingMeasures_{\subregion}$ and $\xi \in \countingMeasures$ it holds that
    \begin{align}
        \hamiltonian_{\subregion}(\eta \mid \xi) = \frac{1}{2} \int_{\subregion} \influence(x, \eta - \delta_x) \eta(\diff x) + \int_{\subregion} \influence(x, \xi_{\subregion^c}) \eta(\diff x). \label{lemma:energy_decomposition:full} 
    \end{align}
    Moreover, for all $x \in \subregion$ it holds that
    \begin{align}
        \hamiltonian_{\subregion}(\eta + \delta_x \mid \xi) = \hamiltonian_{\subregion}(\eta \mid \xi) + \influence(x, \eta + \xi_{\subregion^c}). \label{lemma:energy_decomposition:point}
    \end{align}
\end{lemma}
\begin{proof}
    To prove \eqref{lemma:energy_decomposition:full}, we start by observing that $\eta \in \countingMeasures_{\subregion} \subseteq \finiteCountingMeasures$.
    First, note that
    \[
        \hamiltonian(\eta) = \frac{1}{2} \int_{\subregion} \influence(x, \eta - \delta_x) \eta(\diff x) ,
    \]
    which can be seen by writing both sides as finite sum of terms in $(\infty, \infty]$.
    Moreover, note that $\int_{\subregion} \int_{\subregion^c} \absolute{\potential(x, y)} \xi(\diff y) \eta(\diff x) < \infty$ if and only if $\int_{\subregion^c} \absolute{\potential(x, y)} \xi(\diff y) < \infty$ for every $x \in \support_{\eta}$.
    If this is the case, then 
    \[
        \int_{\subregion} \int_{\subregion^c} \potential(x, y) \xi(\diff y) \eta(\diff x) = \int_{\subregion} \influence(x, \xi_{\subregion^c}) \eta(\diff x) .
    \]
    Otherwise, we have $\int_{\subregion} \influence(x, \xi_{\subregion^c}) \eta(\diff x) = \infty$ and $\hamiltonian_{\subregion}(\eta \mid \xi) = \infty$ by definition, proving identity \eqref{lemma:energy_decomposition:full}.

    For \eqref{lemma:energy_decomposition:point}, we proceed similarly.
    By rewriting $\hamiltonian(\eta + \delta_x)$ and $\influence(x, \eta)$ as a sum of finitely many terms in $(\infty, \infty]$, observing that
    \[
        \hamiltonian(\eta + \delta_x) = \hamiltonian(\eta) + \influence(x, \eta) . 
    \]
    Further, if $\int_{\subregion} \int_{\subregion^c} \absolute{\potential(y, z)} \xi(\diff z) (\eta + \delta_x)(\diff y) < \infty$ then
    \begin{align*}
        \int_{\subregion} \int_{\subregion^c} \potential(y, z) \xi(\diff z) (\eta + \delta_x)(\diff y) 
        &= \int_{\subregion} \int_{\subregion^c} \potential(y, z) \xi(\diff z) \eta(\diff y) + \int_{\subregion^c} \potential(x, z) \xi(\diff z)\\ 
        &= \int_{\subregion} \int_{\subregion^c} \potential(y, z) \xi(\diff z) \eta(\diff y) + \influence(x, \xi_{\subregion^c})
    \end{align*}
    with absolutely converging integrals.
    Hence, in this case we get
    \begin{align*}
        \hamiltonian_{\subregion}(\eta + \delta_x \mid \xi) 
        &= \hamiltonian(\eta) + \influence(x, \eta) + \int_{\subregion} \int_{\subregion^c} \potential(y, z) \xi(\diff z) \eta(\diff y) + \influence(x, \xi_{\subregion^c})\\
        &= \hamiltonian_{\subregion}(\eta \mid \xi) + \influence(x, \eta + \xi_{\subregion^c}) .
    \end{align*}
    Otherwise, if $\int_{\subregion} \int_{\subregion^c} \absolute{\potential(y, z)} \xi(\diff z) (\eta + \delta_x)(\diff y) = \infty$ then, by definition, $\hamiltonian_{\subregion}(\eta + \delta_x \mid \xi) = \infty$.
    Moreover, it must hold that $\int_{\subregion} \int_{\subregion^c} \absolute{\potential(y, z)} \xi(\diff z) \eta(\diff y) = \infty$, which implies $\hamiltonian_{\subregion}(\eta \mid \xi)$, or $\int_{\subregion^c} \absolute{\potential(x, z)} \xi(\diff z) = \infty$, which implies $\influence(x, \eta + \xi_{\subregion^c}) = \infty$.
    In both cases, identity \eqref{lemma:energy_decomposition:point} holds.
\end{proof}
We use \eqref{lemma:energy_decomposition:full} in conjunction with the following lower bound on $\influence(x, \eta)$ to upper-bound the partition function under feasible boundary conditions.
\begin{lemma} \label{lemma:influence_bound}
    Suppose $\potential$ is a locally stable pair potential with constant $\localStability$.
    For all $x \in \groundspace$ and $\eta \in \feasibleCountingMeasures$ it holds that $\influence(x, \eta) \ge - \localStability$.
\end{lemma}
\begin{proof}
    First, we consider the case that $\eta \in \feasibleCountingMeasures \cap \finiteCountingMeasures$ and observe that, by local stability, 
    \[
        \influence(x, \eta) = \hamiltonian(\eta + \delta_x) - \hamiltonian(\eta) \ge - \localStability.
    \]
    Next, suppose that $\eta \in \feasibleCountingMeasures \setminus\finiteCountingMeasures$, and let $f: \groundspace \to \extended{\R_{\ge 0}}, y \mapsto \max(0, -\potential(x, y))$.
    Note that to prove our claim, it suffices to show that $\int f \diff \eta \le \localStability$.
    Let $\eta = \sum_{i \in \N} \delta_{x_i}$ for some sequence $x_i \in \groundspace$.
    For $n \in \N$ define $\eta_n = \sum_{i = 1}^{n} \delta_{x_i}$ and note that $\eta_n \to \eta$ setwise monotonically.
    Therefore, we obtain
    \[
        \int_{\groundspace} f(y) \eta(\diff y) = \lim_{n \to \infty} \int_{\groundspace} f(y) \eta_n(\diff y).
    \]
    Next, for $n\in\N$ let $S_n$ be the set of indices $i \in \{1, \dots, n\}$ such that $\potential(x, x_i) \le 0$ and let $\eta_{S_n} = \sum_{i \in S_n} \delta_{x_i}$.
    Observe that $\eta_{S_n} \in \finiteCountingMeasures$ with $\eta_{s_n} \le \eta$ and, since $\eta \in \feasibleCountingMeasures$, $\hamiltonian(\eta_{S_n}) < \infty$.
    Consequently, by local stability, we have
    \[
        \int_{\groundspace} f(y) \eta_n(\diff y) = \sum_{i=1}^{n} f(x_i) = - \sum_{i \in S_n} \potential(x, x_i) = -(\hamiltonian(\eta_{S_n} + \delta_x) - \hamiltonian(\eta_{S_n})) \le \localStability .   
    \]
    Taking the limit $n \to \infty$ concludes the proof.
\end{proof}

We proceed to show that $\feasibleCountingMeasures$ provides boundary conditions for which finite-volume Gibbs measures are well-defined.
\begin{lemma} \label{lemma:bound_partition_function}
    Suppose $\potential$ is locally stable with constant $\localStability$.
    For all $\subregion \in \boundedBorel$ and $\xi \in \feasibleCountingMeasures$ it holds that $\partitionFunction_{\subregion \mid \xi}(\activity) < \infty$.
    In particular, $\gibbs_{\subregion \mid \xi}$ is well-defined.
\end{lemma}

\begin{proof}
    The main idea is to show that, for all $\eta \in \countingMeasures_{\subregion}$, it holds that $\hamiltonian_{\subregion}(\eta \mid \xi) \ge - C \localStability \eta(\subregion)$ for some absolute constant $C$.
    If $\hamiltonian(\eta) = \infty$ then $\hamiltonian_{\subregion}(\eta \mid \xi) = \infty$.
    Otherwise, we use identity \eqref{lemma:energy_decomposition:full} and \Cref{lemma:influence_bound} to obtain
    \[
        \hamiltonian_{\subregion}(\eta \mid \xi) 
        = \frac{1}{2} \int_{\subregion} \influence(x, \eta - \delta_x) \eta(\diff x) + \int_{\subregion} \influence(x, \xi_{\subregion^c}) \eta(\diff x) 
        \ge - \frac{3}{2} \localStability \eta(\subregion) .
    \]
    Thus, we get
    \begin{align*}
        \partitionFunction_{\subregion \mid \xi} (\activity) \le 1 + \sum_{n = 1}^{\infty} \frac{\activity^n \eulerE^{\frac{3}{2} \localStability n} \volume(\subregion)^n}{n!},
    \end{align*}
    which proves the claim.
\end{proof}

Knowing that $\gibbs_{\subregion \mid \xi}$ is well-defined for every $\xi \in \countingMeasures_{\subregion}$, \eqref{lemma:energy_decomposition:point} can be used to derive the following identity, which was independently proven by Georgii \cite{georgii1976canonical} and Nguyen and Zessin \cite{xanh1979integral}, and which is hence known as GNZ equation.
As some of our definitions differ slightly from those of the original papers, we briefly reprove it.
\begin{lemma}[Finite-volume GNZ equation \cite{georgii1976canonical,xanh1979integral}] \label{lemma:gnz}
    Suppose $\potential$ is locally stable with constant $\localStability$, and let $\subregion \in \boundedBorel$ and $\xi \in \feasibleCountingMeasures$.
    For all non-negative measurable $F: \groundspace \times \countingMeasures \to \extended{\R_{\ge 0}}$ it holds that
    \[
        \int_{\countingMeasures_{\subregion}} \int_{\subregion} F(x, \eta) \eta(\diff x) \gibbs_{\subregion \mid \xi}(\diff \eta) = \int_{\subregion} \int_{\countingMeasures_{\subregion}} \activity \eulerE^{- \influence(x, \eta + \xi_{\subregion^c})}  F(x, \eta + \delta_x) \gibbs_{\subregion \mid \xi}(\diff \eta) \volume(\diff x).
    \]
\end{lemma}
\begin{proof}
    Denoting by $Q_{\subregion}$ a Poisson point process with intensity measure $\volume(\cdot \cap \subregion)$, we get
    \begin{align*}
        \int_{\countingMeasures_{\subregion}} \int_{\subregion} F(x, \eta) \eta(\diff x) \gibbs_{\subregion \mid \xi}(\diff \eta) 
        &= \frac{\eulerE^{\volume(\subregion)}}{\partitionFunction_{\subregion \mid \xi} (\activity)}\int_{\countingMeasures_{\subregion}} \int_{\subregion} F(x, \eta) \eta(\diff x) \activity^{\eta(\subregion)} \eulerE^{- \hamiltonian_{\subregion}(\eta \mid \xi)}  Q_{\subregion}(\diff \eta) \\
        &= \int_{\countingMeasures_{\subregion}} \int_{\subregion} G(x, \eta) \eta(\diff x) Q_{\subregion}(\diff \eta)
    \end{align*}
    for $G(x, \eta) \coloneqq \frac{\eulerE^{\volume(\subregion)}}{\partitionFunction_{\subregion \mid \xi} (\activity)} F(x, \eta) \activity^{\eta(\subregion)} \eulerE^{- \hamiltonian_{\subregion}(\eta \mid \xi)}$. 
    Applying the Mecke equation (see \cite[Proposition 4.1]{last2017lectures}) and the decomposition in \eqref{lemma:energy_decomposition:point} then yields
    \begin{align*}
        \int_{\countingMeasures_{\subregion}} \int_{\subregion} G(x, \eta) \eta(\diff x) Q_{\subregion}(\diff \eta)
        &= \int_{\subregion} \int_{\countingMeasures_{\subregion}} G(x, \eta + \delta_x) Q_{\subregion}(\diff \eta) \volume(\diff x) \\
        &= \int_{\subregion} \frac{\eulerE^{\volume(\subregion)}}{\partitionFunction_{\subregion \mid \xi} (\activity)}\int_{\countingMeasures_{\subregion}} \activity \eulerE^{- \influence(x, \eta + \xi_{\subregion^c})}  F(x, \eta + \delta_x) \activity^{\eta(\subregion)} \eulerE^{- \hamiltonian_{\subregion}(\eta \mid \xi)} Q_{\subregion}(\diff \eta) \volume(\diff x) \\
        &= \int_{\subregion} \int_{\countingMeasures_{\subregion}} \activity \eulerE^{- \influence(x, \eta + \xi_{\subregion^c})}  F(x, \eta + \delta_x) \gibbs_{\subregion \mid \xi}(\diff \eta) \volume(\diff x),
    \end{align*}
    which concludes the proof.
\end{proof}

Next, we introduce a restriction of $\countingMeasures_{\subregion}$ to configurations that are consistent with a given boundary condition $\xi \in \countingMeasures$ by
\begin{align}
    \countingMeasures_{\subregion \mid \xi} \coloneqq \{\eta \in \countingMeasures_{\subregion} \mid \eta + \xi_{\subregion^c} \in \feasibleCountingMeasures \} .
    \label{eq:consistent}
\end{align}
Using similar arguments as for $\feasibleCountingMeasures$, it can be proven that $\countingMeasures_{\subregion \mid \xi} \in \countingAlgebra$. 
We denote the trace of $\countingMeasures_{\subregion \mid \xi}$ in $\countingAlgebra$ by $\countingAlgebra_{\subregion \mid \xi}$.
Rather than studying the finite-volume Gibbs point process $\gibbs_{\subregion \mid \xi}$ on $(\countingMeasures_{\subregion}, \countingAlgebra_{\subregion})$, we will restrict it to $(\countingMeasures_{\subregion \mid \xi}, \countingAlgebra_{\subregion \mid \xi})$, which is justified for every $\xi \in \feasibleCountingMeasures$ due to the following lemma.

\begin{lemma} \label{lemma:gibbs_support}
    Suppose $\potential$ is locally stable with constant $\localStability$, and let $\subregion \in \boundedBorel$. For every $\xi \in \feasibleCountingMeasures$ it holds that $\gibbs_{\subregion \mid \xi}(\countingMeasures_{\subregion \mid \xi}) = 1$.
\end{lemma}

\begin{proof}
    We show that $\hamiltonian_{\subregion}(\eta \mid \xi) = \infty$ for all $\eta \in \countingMeasures_{\subregion} \setminus \countingMeasures_{\subregion \mid \xi}$.
    If $\hamiltonian(\eta)=\infty$ this obviously holds.
    Hence, assume $\hamiltonian(\eta) < \infty$.
    By definition of $\countingMeasures_{\subregion \mid \xi}$ and the fact that $\hamiltonian$ is hereditary, we can find some $\zeta \in \finiteCountingMeasures$, $\zeta \le \xi_{\subregion^c}$ such that $\hamiltonian(\eta + \zeta) = \infty$.
    Since $\xi \in \feasibleCountingMeasures$, it further holds that $\hamiltonian(\zeta) < \infty$. 
    Next, note that
    \[
        \hamiltonian(\eta + \zeta) = \hamiltonian(\eta) + \hamiltonian(\zeta) + \int_{\subregion} \int_{\subregion^c} \potential(x, y) \zeta(\diff y) \eta(\diff x).
    \]
    Hence, there must be some $x \in \support_{\eta}, y \in \support_{\zeta}$ such that $\potential(x, y) = \infty$.
    But then, since $\zeta \le \xi_{\subregion^c}$, we know that $\int_{\subregion^c} \absolute{\potential(x, y)} \xi(\diff y) \eta(\diff x) = \infty$ and hence $\hamiltonian_{\subregion}(\eta \mid \xi) = \infty$.
    Thus, we have
    \[
        \gibbs_{\subregion \mid \xi}(\countingMeasures_{\subregion \mid \xi}^c) 
        = \int_{\countingMeasures_{\subregion} \setminus\countingMeasures_{\subregion \mid \xi}} \activity^{\eta(\subregion)} \eulerE^{-\hamiltonian_{\subregion}(\eta \mid \xi)} Q_{\subregion}(\diff \eta) 
        = 0,  
    \]
    which concludes the proof.
\end{proof}

Finally, we introduce the main lemma that we use to show uniqueness of infinite-volume Gibbs point processes.
It is a simple criterion for equivalence of infinite-volume Gibbs measures that is derived from the DLR formalism.
\begin{lemma} \label{lemma:uniqueness_from_ssm}
    Let $\gibbs_1, \gibbs_2 \in \gibbsMeasures(\activity, \potential)$ for an activity $\activity \in \R_{\ge 0}$ and potential $\potential$.
    Suppose there is a non-decreasing sequence $(\subregion_{k})_{k \in \N}$ in $\boundedBorel$, such that
    \begin{enumerate}
        \item for every $\Delta \in \boundedBorel$ it holds that $\Delta \subseteq \subregion_k$ for some $k \in \N$, and \label{lemma:uniqueness_from_ssm:cover}
        \item for every $k \in \N$ it holds that
        \[
            \lim_{n \to \infty} \esssup_{\xi_1 \sim \gibbs_1, \xi_2 \sim \gibbs_2} \dtv{\gibbs_{\subregion_{k+n} \mid \xi_1}[\subregion_{k}]}{\gibbs_{\subregion_{k+n} \mid \xi_2}[\subregion_{k}]} = 0,
        \]
        where $\dtv{\cdot}{\cdot}$ denotes the total variation distance, and the essential supremum is taken over pairs $(\xi_1, \xi_2) \in \countingMeasures \times \countingMeasures$ and with respect to the product measure $\gibbs_1 \otimes \gibbs_2$. \label{lemma:uniqueness_from_ssm:dtv}
    \end{enumerate}
    Then $\gibbs_1 = \gibbs_2$.
\end{lemma}

\begin{proof}
    By \Cref{fact:equality}, it suffices to show that $\gibbs_1(A) = \gibbs_2(A)$ for all $A \in \countingAlgebra$ of the form 
    \[
        A = \{\pointcount_{\Delta_1} = m_1, \dots, \pointcount_{\Delta_{\ell}} = m_{\ell}\}
    \]
    for $\ell \in \N$, $\Delta_1, \dots, \Delta_{\ell} \in \boundedBorel$ and $m_1, \dots, m_{\ell} \in \N_0$.
    By assumption, we can find some $k \in \N$ such that $\Delta_j \subseteq \subregion_k$ for all $1 \le j \le \ell$.
    This implies $\eta(\Delta_j) = \eta_{\subregion_{k}}(\Delta_j)$ for all $\eta \in \countingMeasures$ and all $1 \le j \le \ell$.
    In particular, for all $n \in \N_0$, all $\eta \in \countingMeasures_{\subregion_{k + n}}$ and all $\xi \in \countingMeasures$, we have that $\eta + \xi_{\groundspace \setminus \subregion_{k + n}} \in A$ if and only if $\eta_{\subregion_k} \in A$.
    By the DLR formalism, we obtain for all $n \in \N_{0}$
    \begin{align*}
        \gibbs_1(A) 
        = \int_{\countingMeasures} \int_{\subregion_{k + n}} \ind{\eta_{\subregion_k} \in A} \gibbs_{\subregion_{k + n} \mid \xi}(\diff \eta) \gibbs_1(\diff \xi) 
        = \int_{\countingMeasures} \gibbs_{\subregion_{k + n} \mid \xi}[\subregion_k](A) \gibbs_1(\diff \xi) .
    \end{align*}
    Applying the same argument to $\gibbs_2$ yields 
    \begin{align*}
        \absolute{\gibbs_1(A) - \gibbs_2(A)} 
        &\le \int_{\countingMeasures} \int_{\countingMeasures} \absolute{\gibbs_{\subregion_{k + n} \mid \xi_1}[\subregion_k](A) - \gibbs_{\subregion_{k + n} \mid \xi_2}[\subregion_k](A)} \gibbs_1(\diff \xi_1) \gibbs_2(\diff \xi_2) \\
        &\le \int_{\countingMeasures} \int_{\countingMeasures} \dtv{\gibbs_{\subregion_{k + n} \mid \xi_1}[\subregion_k]}{\gibbs_{\subregion_{k + n} \mid \xi_2}[\subregion_k]} \gibbs_1(\diff \xi_1) \gibbs_2(\diff \xi_2) \\
        &\le \esssup_{\xi_1 \sim \gibbs_1, \xi_2 \sim \gibbs_2} \dtv{\gibbs_{\subregion_{k+n} \mid \xi_1}[\subregion_{k}]}{\gibbs_{\subregion_{k+n} \mid \xi_2}[\subregion_{k}]} .
    \end{align*}
    Since this holds for all $n$, taking the limit $n \to \infty$ shows $\gibbs_1(A) = \gibbs_2(A)$, concluding the proof.
\end{proof}

\section{Jump processes} \label{sec:jump_process}

\subsection{Terminology and notation}
We recall basic terminology related to kernel and transition groups.
Given a measurable space $(E, \mathcal{E})$, a \emph{kernel} on $(E, \mathcal{E})$ is a map $K: E \times \mathcal{E} \to \extended{\R_{\ge 0}}$ such that:
\begin{enumerate}
    \item For all $x \in E$, $K(x, \cdot)$ is a measure on $(E, \mathcal{E})$.
    \item For all $A \in \mathcal{E}$, $K(\cdot, A)$ is measurable.
\end{enumerate}
Further we say a kernel $K$ is
\begin{itemize}
    \item \emph{finite} if $K(x, \cdot)$ is a finite measure for all $x \in E$,
    \item \emph{sub-stochastic} if $K(x, E) \le 1$ for all $x \in E$, and
    \item \emph{stochastic} if $K(x, E) = 1$ for all $x \in E$.
\end{itemize}
In particular, if $K$ is stochastic then $K(x, \cdot)$ is a probability measure for all $x \in E$.

Occasionally, it will be useful to associate a kernel with an operator on the space of non-negative measurable functions $f: E \to \extended{\R_{\ge 0}}$.
More precisely, we define the function $Kf: E \to \extended{\R_{\ge 0}}$ by setting
\[
    (Kf)(x) \coloneqq \int_E f(y) K(x, \diff y) 
\]
for $x \in E$.
Note that $Kf$ is again measurable (this can be seen by taking monotone limits of simple functions).

Based on these definitions, a \emph{sub-Markov transition family} on $(E, \mathcal{E})$ is a collection $(P_t)_{t \in \R_{\ge 0}}$ such that
\begin{enumerate}
    \item For every $t \in \R_{\ge 0}$, $P_t$ is a sub-stochastic kernel.
    \item For all $x \in E$ and $A \in \mathcal{E}$ we have $P_0(x, A) = \delta_x (A)$.
    \item $(P_t)_{t \in \R_{\ge 0}}$ satisfies the Chapman--Kolmogorov equation: for all $t, s \in \R_{\ge 0}$ and all $x \in E$ , $A \in \mathcal{E}$ it holds that
    \[
        P_{t+s}(x, A) = \int_E P_s (y, A) P_t(x, \diff y) . 
    \]
\end{enumerate}
When talking about a sub-Markov transition family, we will usually drop the index set and simply write $(P_t)$ instead of $(P_t)_{t \in \R_{\ge 0}}$.
Further, we say that $(P_t)$ is a \emph{Markov transition family} if $P_t$ is stochastic for every $t \in \R_{\ge 0}$.

\subsection{Jump kernel and associated transition families} 
We proceed by introducing the basic terminology related to Markov jump processes with values in a general measurable space $(E, \mathcal{E})$.
We will always assume that $\mathcal{E}$ contains all singletons (i.e., $\{x\} \in \mathcal{E}$ for all $x \in E$).

Throughout the paper, we mostly characterize Markov jump processes by their \emph{jump kernel}.
A function $K: E \times \mathcal{E} \to \R_{\ge 0}$ is a jump kernel on $(E, \mathcal{E})$ if:
\begin{enumerate}
    \item $K$ is a finite kernel.
    \item For all $x \in E$ we have $K(x, \{x\}) = 0$.
\end{enumerate}

Given a jump kernel $K$, we call $\kappa: E \to \R_{\ge 0}, x \mapsto K(x, E)$ the \emph{associated rate function}.
Further, note that $\Pi: E \times \mathcal{E} \to [0, 1]$ with
\begin{align} \label{eq:transition_kernel}
    \Pi(x, A) \coloneqq \begin{cases}
        \frac{1}{\kappa(x)} \cdot K(x, A) &\text{ if $\kappa(x) > 0$} \\
        \delta_x(A) &\text{ otherwise} 
    \end{cases} 
\end{align}
defines a Markov kernel on $(E, \mathcal{E})$, which we call the \emph{associated transition kernel}. 

Given a jump kernel $K$ with jump rate function $\kappa$, we recursively define a sub-Markov transition family $(P^{(n)}_{t})$ for every $n \in \N_{0}$ as follows. 
For every $x \in E$, $A \in \mathcal{E}$ and $t \in \R_{\ge 0}$ we set 
\begin{gather}
    P_{t}^{(0)}(x, A) = \eulerE^{- \kappa(x) t} \cdot \delta_{x}(A). \label{eq:transition_family_base}
\end{gather}
Further, for $n \ge 1$, we set
\begin{gather}
    P_{t}^{(n)} (x, A) = \eulerE^{-\kappa(x) t} \cdot \delta_{x}(A) + \int_{0}^{t} \eulerE^{- \kappa(x) s} \int_{E} P_{t-s}^{(n-1)}(y, A) K(x, \diff y) \diff s . \label{eq:transition_family_step}
\end{gather}
The following properties can easily be checked via induction over $n$ (see for example \cite{feller1991introduction}):
\begin{enumerate}
    \item For every $n \in \N_{0}$, $(P_{t}^{(n)})$ is a well-defined sub-Markov transition family, and for every $A \in \mathcal{E}$ the map $(t, x) \mapsto P_{t}^{(n)}(x, A)$ is measurable.
    \item For every $x \in E$, $A \in \mathcal{E}$ and $t \in \R_{\ge 0}$, it holds that $P_{t}^{(n)}(x, A)$ is non-decreasing in $n$.
\end{enumerate}
It immediately follows by monotone convergence that the limit $P_{t}(x, A) = \lim_{n \to \infty} P_{t}^{(n)}(x, A)$ exists and defines a sub-Markov transition family $(P_t)$.
We call $(P_t)$ the \emph{sub-Markov transition family associated to} $K$.
Note that $(P_t)$ is in fact Markov whenever $P_t(x, E) = 1$ for all $x \in E$ and $t \in \R_{\ge 0}$.
In that case, we say that the jump kernel $K$ is \emph{non-explosive} and simply call $(P_t)$ the \emph{transition family associated to} $K$.

To check if a given jump kernel is non-explosive, we are going to use the following simple criterion. 
\begin{lemma}[{\cite[Theorem $16$]{mufa1986coupling}}] \label{lemma:nonexp_test}
    Let $K$ be a jump kernel on $(E, \mathcal{E})$ and let $\kappa$ denote its rate function.
    Suppose there exists a sequence $(E_n)_{n \in \N}$ with $E_n \in \mathcal{E}$ and a non-negative measurable function $f: E \to \R_{\ge 0}$ such that the following holds:
    \begin{enumerate}
        \item $E_n$ converges to $E$ monotonically.
        \item For all $n \in \N$ we have $\sup_{x \in E_n} \kappa(x) < \infty$.
        \item $\lim_{n \to \infty} \inf_{x \in E \setminus E_n} f(x) = \infty$.
        \item There exists some $c \in \R$ such that for all $x \in E$ we have
        \[
            (Kf)(x) \le (c + \kappa(x)) \cdot f(x)
        \]
    \end{enumerate}
    Then $K$ is non-explosive.
\end{lemma}

\subsection{Stationary distribution and reversibility}
Given a Markov transition family $(P_t)$ on $(E, \mathcal{E})$, we call a probability distribution $\pi$ on $(E, \mathcal{E})$ a \emph{stationary distribution} (or \emph{invariant probability measure}) for $(P_t)$ if for all $t \in \R_{\ge 0}$ and $A \in \mathcal{E}$ it holds that
\[
    \int_{E} P_t(x, A) \pi(\diff x) = \pi(A) . 
\]
A sufficient criterion for $\pi$ to be a stationary distribution is if for all $t \in \R_{\ge 0}$ and $A, B \in \mathcal{E}$ it holds that
\[
    \int_{A} P_t(x, B) \pi(\diff x) = \int_{B} P_t(x, A) \pi(\diff x), 
\]
in which case we say that $(P_t)$ is \emph{reversible} with respect to $\pi$.
Note that the definition of reversibility extends to sub-Markov transition families.
However, it only implies that $\pi$ is an invariant measure if $(P_t)$ is stochastic (hence Markov).

If $(P_t)$ is generated by a jump kernel $K$, reversible measures can be identified using $K$.
This is given by the following lemma, which we believe was first proven by Chen \cite{mufa1980}.
\begin{lemma}[{\cite[Proposition 2]{serfozo2005reversible}}] \label{lemma:reversible_kernel}
    Let $K$ be a jump kernel on $(E, \mathcal{E})$ with associated (sub-)Markov transition family $(P_t)$.
    If a probability measure $\pi$ on $(E, \mathcal{E})$ satisfies 
    \[
        \int_{A} K(x, B) \pi(\diff x) = \int_{B} K(x, A) \pi(\diff x)
    \]
    for all $A, B \in \mathcal{E}$ then $(P_t)$ is reversible with respect to $\pi$.
    Further, if $K$ is non-explosive then $\pi$ is invariant under $(P_t)$.
\end{lemma}

\subsection{Coupling and convergence}
Given probability measures $\pi_1, \pi_2$ on $(E_1, \mathcal{E}_1)$ and $(E_2, \mathcal{E}_2)$, a coupling of $\pi_1, \pi_2$ is a probability measure $\pi$ on the product space $(E_1 \times E_2, \mathcal{E}_1 \otimes \mathcal{E}_2)$ with $\pi(A_1 \times E_2) = \pi_1(A_1)$ for all $A_1 \in \mathcal{E}_1$ and $\pi(E_1 \times A_2) = \pi_2(A_2)$ for all $A_2 \in \mathcal{E}_2$.
A useful aspect of couplings is that they provide a way of bounding the total variation distance between probability distributions.
\begin{lemma}[{\cite[Theorem 2.4]{den2012probability}}] \label{lemma:coupling_ineq}
    Let $\pi_1, \pi_2$ be probability measures on $(E, \mathcal{E})$.
    Suppose $(E, \mathcal{E})$ is such that $\mathcal{E} \otimes \mathcal{E}$ contains the diagonal set $D \coloneqq \{(x,x) \mid x \in E\}$ then, for every coupling $\pi$ of $\pi_1$ and $\pi_2$, it holds that
    \[
        \dtv{\pi_1}{\pi_2} \le 2 \cdot \pi(D^c),
    \]
    where $\dtv{\pi_1}{\pi_2} \coloneqq \sup_{A \in \mathcal{E}} \absolute{\pi_1(A) - \pi_2(A)}$ denotes the total variation distance between $\pi_1$ and $\pi_2$.
\end{lemma}

The notion of a coupling extends directly to Markov transition families as follows.
\begin{definition} \label{def:coupling_markov}
    Suppose $(P_t)$, $(P'_t)$ are Markov transition families on $(E_1, \mathcal{E}_1)$ and $(E_2, \mathcal{E}_2)$.
    We call a Markov transition family $(Q_t)$ on $(E_1 \times E_2, \mathcal{E}_1 \otimes \mathcal{E}_2)$ a coupling of $(P_t)$, $(P'_t)$ if for all $x \in E_1, y \in E_2$ and $t \in \R_{\ge 0}$ it holds that $Q_t((x, y), \cdot)$ is a coupling of $P_t(x, \cdot)$ and $P'_t(y, \cdot)$.
\end{definition}

We wish to construct couplings between (jump) Markov transition families based on their jump kernels.
Following the ideas in Chen \cite{mufa1986coupling}, we introduce the following definition.
\begin{definition} \label{def:coupling_kernel}
    Let $K_1, K_2$ be jump kernels on $(E_1, \mathcal{E}_1)$ and $(E_2, \mathcal{E}_2)$.
    We call a jump kernel $K$ on $(E_1 \times E_2, \mathcal{E}_1 \otimes \mathcal{E}_2)$ a coupling of $K_1$ and $K_2$ if for all non-negative bounded measurable functions $f_1: E_1 \to \R_{\ge 0}, f_2: E_2 \to \R_{\ge 0}$ and all $x_1 \in E, x_2 \in E_2$ it holds that
    \begin{align*}
        \int_{E_1 \times E_2} (f_1(y_1) - f_1(x_1)) K((x_1, x_2), \diff y_1 \times \diff y_2) = (K_1 f_1)(x_1) - \kappa_1(x_1) \cdot f_1(x_1) \\
        \int_{E_1 \times E_2} (f_2(y_2) - f_2(x_2)) K((x_1, x_2), \diff y_1 \times \diff y_2) = (K_2 f_2)(x_2) - \kappa_2(x_2) \cdot f_2(x_2) .
    \end{align*}
\end{definition}

The following result relates the two notions of a coupling that are given in \Cref{def:coupling_markov,def:coupling_kernel}.
\begin{lemma}[{\cite[Theorems~$13$ and~$25$]{mufa1986coupling}}] \label{lemma:coupling}
    Let $K_1, K_2$ be non-explosive jump kernels on $(E_1, \mathcal{E}_1)$ and $(E_2, \mathcal{E}_2)$.
    If $K$ is a coupling of $K_1, K_2$ in the sense of \Cref{def:coupling_kernel} then $K$ is non-explosive, and the Markov transition family associated to $K$ is a coupling of the Markov transition families associated to $K_1, K_2$ in the sense of \Cref{def:coupling_markov}.
\end{lemma}

A central part of this paper will be to bound the total variation distance of certain jump Markov transition families from their stationary distribution (provided it exists) after a sufficiently large time span $t$.
The following technical lemma will be useful for that.
\begin{lemma} \label{lemma:potential}
    Let $K$ be a jump kernel on $(E, \mathcal{E})$ with jump rate function $\kappa$ and (sub-)Markov transition family $(P_t)$.
    Suppose $f: E \to \R_{\ge 0}$ is a non-negative measurable function such that there exists some $\delta > 0$ such that for all $x \in E$ it holds that
    \[
        (Kf)(x) \le (\kappa(x) - \delta) \cdot f(x) .
    \]
    Then, for all $x \in E$ and $t \in \R_{\ge 0}$, we have 
    \[
        (P_{t}f)(x) \le \eulerE^{- \delta t} f(x) . \qedhere
    \]
\end{lemma}

\begin{proof}
    We show by induction that for all $n \in \N_{0}$ it holds that
    \[
        (P_{t}^{(n)} f)(x) \le \eulerE^{- \delta t} f(x) ,
    \]
    where $(P^{(n)}_{t})$ is as defined in \eqref{eq:transition_family_base} and \eqref{eq:transition_family_step}.
    The desired result then follows from the fact that, for all $t \in \R_{\ge 0}$ and $x \in E$, $P_t^{(n)}(x, \cdot) \to P_t(x, \cdot)$ setwise and monotonically as $n \to \infty$.

    For the base case $n=0$, we have by \eqref{eq:transition_family_base} that
    \[
        (P_{t}^{(0)} f)(x) = \eulerE^{- \kappa(x) t} f(x) .
    \]
    Hence, we need to show $\eulerE^{- \kappa(x) t} f(x) \le \eulerE^{- \delta t} f(x)$ for all $x \in E$.
    Since this holds trivially if $f(x) = 0$, we focus on $x \in  E$ such that $f(x) > 0$.
    Because $f$ is non-negative, our assumptions imply $0 \le (Kf)(x) \le (\kappa(x) - \delta) \cdot f(x)$.
    Thus, we have $\delta \le \kappa(x)$ for all $x \in E$ with $f(x) > 0$ and the desired inequality follows.

    For the induction step, combining \eqref{eq:transition_family_step} and the induction hypothesis yields
    \begin{align*}
        (P_{t}^{(n)} f)(x) &= \eulerE^{-\kappa(x) t} \cdot f(x) + \int_{0}^{t} \eulerE^{- \kappa(x) s} \int_{E} (P_{t-s}^{(n-1)} f)(y) K(x, \diff y) \diff s \\
        &\le \eulerE^{-\kappa(x) t} \cdot f(x) + \int_{0}^{t} \eulerE^{- \kappa(x) s} \int_{E} \eulerE^{-\delta \cdot (t - s)} f(y) K(x, \diff y) \diff s \\
        &= \eulerE^{-\kappa(x) t} \cdot f(x) + \eulerE^{-\delta t} \int_{0}^{t} \eulerE^{- (\kappa(x) - \delta) \cdot s} (Kf)(x) \diff s .
    \end{align*}
    Using $(Kf)(x) \le (\kappa(x) - \delta) \cdot f(x)$, we obtain
    \begin{align*}
        (P_{t}^{(n)} f)(x) &\le \left( \eulerE^{-\kappa(x) t} + \eulerE^{-\delta t} \int_{0}^{t} \eulerE^{- (\kappa(x) - \delta) \cdot s} (\kappa(x) - \delta) \diff s\right) \cdot f(x) \\
        &= \left( \eulerE^{-\kappa(x) t} + \eulerE^{-\delta t} \left(1 - \eulerE^{-(\kappa(x) - \delta) \cdot t} \right) \right) \cdot f(x) \\
        &= \eulerE^{-\delta t} f(x) ,
    \end{align*}
    which concludes the induction step and proves the claim.
\end{proof}

We can now combine \Cref{lemma:potential} with a suitable coupling of a non-explosive jump kernel to derive information about the speed of convergence of the associated Markov transition family to its stationary distribution.

\begin{lemma} \label{lemma:mixing}
    Let $K$ be a non-explosive jump kernel on $(E, \mathcal{E})$ and assume the associated Markov transition group $(P_t)$ has a stationary distribution $\pi$.
    Suppose there is a coupling $\Tilde{K}$ of $K$ with itself (in the sense of \Cref{def:coupling_kernel}) and a non-negative measurable function $f: E \times E \to \R_{\ge 0}$ such that the following hold:
    \begin{enumerate}
        \item There exists some $c > 0$ such that $f(x, y) \ge c$ if and only if $x \neq y$.
        \label{item:mixing:threshold}
        \item There exists some $\delta > 0$ such that for all $x, y \in E$
        \[
            (\Tilde{K} f)(x, y) \le (\Tilde{\kappa}(x, y) - \delta) \cdot f(x, y) ,
        \]
        where $\Tilde{\kappa}$ is the rate function of $\Tilde{K}$. \label{item:mixing:contraction}
    \end{enumerate}
    Then, for all $x \in E$ and all $t \in \R_{\ge 0}$, it holds that
    \[
        \dtv{P_t (x, \cdot)}{\pi} \le \frac{1}{c} \cdot \eulerE^{- \delta t} \int_{E} f(x, y) \pi(\diff y). \qedhere
    \]
\end{lemma}

\begin{proof}
    By \Cref{lemma:coupling} we know that $\Tilde{K}$ is non-explosive and that the associated Markov transition family $(Q_t)$ is a coupling of $(P_t)$ with itself in the sense of \Cref{def:coupling_markov}.
    Further, by (\ref{item:mixing:contraction}) and \Cref{lemma:potential}, we have for all $x, y \in E$ and all $t \in \R_{\ge 0}$ that $(Q_t f)(x, y) \le \eulerE^{- \delta t} f(x, y)$.

    Next, for $x \in E$, $t \in \R_{\ge 0}$ and $A \in \mathcal{E} \otimes \mathcal{E}$ define
    \[
        Q_{x, t}(A) \coloneqq \int_{E \times E} Q_t(z, A) (\delta_x \otimes \pi)(\diff z),
    \]
    where $\delta_x \otimes \pi$ is the product probability measure of $\delta_x$ and $\pi$.
    Note that for all $x$ and $t$ it holds that $Q_{x, t}$ is a probability measure on $(E \times E, \mathcal{E} \otimes \mathcal{E})$. 
    Using Fubini's theorem and the fact that $(Q_t)$ is a coupling of $(P_t)$ with itself, we obtain for all $A \in \mathcal{E}$
    \begin{align*}
         Q_{x, t}(A \times E) 
         &= \int_E \int_E Q_t((z_1, z_2), A \times E) \delta_x(\diff z_1) \pi(\diff z_2) \\
         &= \int_E \int_E P_t(z_1, A) \delta_x(\diff z_1) \pi(\diff z_2) \\
         &= \int_E P_t(z_1, A) \delta_x (\diff z_1) \\
         & = P_t(x, A).
    \end{align*}
    Further, since $\pi$ is invariant with respect to $(P_t)$, we have analogously
    \[
        Q_{x, t}(E \times A) 
        = \int_E P_t(z_2, A) \pi(\diff z_2)
        = \pi(A) . 
    \]
    This shows that $Q_{x, t}$ is a coupling of $P_t(x, \cdot)$ and $\pi$.
    Note that by (\ref{item:mixing:threshold}) we may express the diagonal as $D = \{(x, x) \mid x \in E\} = \{(x, y) \in E \times E \mid f(x, y) < c\}$ and, since $f$ is measurable, $D \in \mathcal{E} \otimes \mathcal{E}$.
    By \Cref{lemma:coupling_ineq}, it suffices to bound $Q_{x, t} (D^c)$.
    Applying Markov's inequality, we get
    \[
         \dtv{P_t(x, \cdot)}{\pi} 
         \le Q_{x, t}(D^c) \le \frac{1}{c} \int_{E \times E} f(z) Q_{x, t}(\diff z)
         = \frac{1}{c} \int_{E \times E} (Q_t f)(z) (\delta_x \otimes \pi) (\diff z).
    \]
    Finally, since $(Q_t f)(z) \le \eulerE^{- \delta t} f(z)$, applying Fubini's theorem yields
    \[
        \dtv{P_t(x, \cdot)}{\pi} 
        \le \frac{1}{c} \cdot \eulerE^{- \delta t} \int_{E \times E} f(z) (\delta_x \otimes \pi) (\diff z)
        = \frac{1}{c} \cdot \eulerE^{- \delta t} \int_{E} f(x, y) \pi (\diff y) ,
    \]
    which concludes the proof.
\end{proof}

\subsection{Markov process associated with a jump kernel}
We proceed by discussing in what sense we can associate a stochastic process (more precisely a Markov jump process) on $(E, \mathcal{E})$ with a given jump kernel.
By a stochastic process we mean a tuple $(\Omega, \filtration, \Pr, (X_t)_{t \in \R_{\ge 0}})$ where $(\Omega, \filtration, \Pr)$ is a probability space and $(X_t)_{t \in \R_{\ge 0}}$ is a collection of random variables $X_t: \Omega \to E$.
Since we will always use $\R_{\ge 0}$ as the index set for $(X_t)$, we will drop it from now on to simplify notation (the same will hold for filtrations as introduced below).
For every $\omega \in \Omega$, we call the function $t \mapsto X_t(\omega)$ the \emph{sample path} of $\omega$.
We call a stochastic process \emph{measurable} if the map $(t, \omega) \mapsto X_t(\omega)$ is measurable with respect to $\borel_{\R_{\ge 0}} \otimes \filtration$ and $\mathcal{E}$, where $\borel_{\R_{\ge 0}}$ is the Borel $\sigma$-field on $\R_{\ge 0}$.

A \emph{filtration} on $(\Omega, \filtration, \Pr, (X_t))$ is a sequence of $\sigma$-fields $(\filtration_t)$ such that $\filtration_s \subseteq \filtration_t \subseteq \filtration$ for all $s, t \in \R_{\ge 0}$, $s \le t$. 
We call a random variable $T: \Omega \to \extended{\R_{\ge 0}}$ a stopping time with respect to a filtration $(\filtration_t)$ if, for all $t \in \R_{\ge 0}$, the event $\{T \le t\}$ is measurable with respect to $\filtration_t$.
Given a stopping time $T$ with respect to some filtration $(\mathcal{F}_t)$, we define $\mathcal{F}_T \coloneqq \{A \in \mathcal{F} \mid \forall t \in \R_{\ge 0}: A \cap \{T \le t\} \in \mathcal{F}_t\}$.
Further, say that a stochastic process is \emph{progressive} with respect to a filtration $(\filtration_t)$ if for all $t \in \R_{\ge 0}$ it holds that $(s, \omega) \mapsto X_s(\omega)$ as a function from $[0, t] \times \Omega$ into $E$ is measurable with respect to $\borel_{[0, t]} \otimes \filtration_t$ and $\mathcal{E}$. 
A particularly important filtration is the \emph{natural filtration} $(\filtration^{X}_{t})$, which is given by $\filtration^{X}_t \coloneqq \sigma(X_s \mid s \in [0, t])$.

Finally, given a probability measure $\mu$ and a Markov transition family $(P_t)$ on $(E, \mathcal{E})$, we say that $(\Omega, \filtration, \Pr,(X_t))$ is a \emph{Markov process} with \emph{initial distribution} $\mu$ and \emph{transition function} $(P_t)$ if: 
\begin{enumerate}
    \item for all $A \in \mathcal{E}$ we have $\Pr[X_0 \in A] = \mu(A)$, and
    \item for all $t, s \in \R_{\ge 0}$ and all $A \in \mathcal{E}$ it holds that $P_t(X_s, A)$ is a version of the conditional probability $\Pr[X_{s + t} \in A \mid \filtration^{X}_{s}]$.
\end{enumerate}
Note that, for every $t \in \R_{\ge 0}$ and $A \in \mathcal{E}$, it holds that $x \mapsto P_t(x, A)$ is measurable. Hence, it also holds that $\omega \mapsto P_t(X_s(\omega), A)$ is measurable with respect to $\sigma(X_s)$ and thus 
\[
    \Pr[X_{s + t} \in A \mid \filtration^{X}_{s}] = \Pr[X_{s + t} \in A \mid X_s] \text{\hspace{3em} $\Pr$-almost surely.}
\]
This is usually referred to as the Markov property.

It turns out that, given a non-explosive jump kernel $\jumpKernel$ and associated transition family $(P_t)$, we can construct a particularly nice version of the Markov process with transition function $(P_t)$ for any given initial distribution (see for example \cite[§1.12]{BG} and \cite{jansen2020markov}).

\begin{theorem}
    \label{thm:construction}
    Let $K$ be a non-explosive jump kernel and let $\mu$ be a probability measure on $(E, \mathcal{E})$.
    There is a stochastic process $(\Omega, \mathcal{F}, \Pr_{\mu}, (X_t))$ with values in $(E, \mathcal{E})$ such that:
    \begin{enumerate}[series=conditions]
        \item The sample paths of $(X_t)$ are piecewise constant.
        That is, for every $\omega \in \Omega$ and $s \in \R_{\ge 0}$ there is some $h > 0$ such that the map $t \mapsto X_t(\omega)$ is constant on $[s, s+h)$. \label{thm:construction:sample_paths}
        \item For every $A \in \mathcal{E}$, $T_A \coloneqq \inf\{t \in \R_{\ge 0} \mid X_t \in A\}$ is a stopping time with respect to $(\mathcal{F}^{X}_{t})$.
        \item $(\Omega, \mathcal{F}, \Pr_{\mu}, (X_t))$ is a Markov process with initial distribution $\mu$ and transition function $(P_t)$. \label{thm:construction:Markov}
        \item The process satisfies the following version of the \emph{strong Markov property}: Suppose $T$ is a stopping time with respect to $(\mathcal{F}^{X}_t)$, then for every $t \in \R_{\ge 0}$ and $A \in \mathcal{E}$
        \[
            \Pr_{\mu}[T < \infty, X_{T+t} \in A \mid \mathcal{F}^{X}_{T}] = \ind{T<\infty} \cdot P_{t}(X_T, A) \text{ $\Pr_{\mu}$-a.s.}
        \] 
        \label{thm:construction:strongMarkov}
    \end{enumerate}
    Moreover, on the same probability space $(\Omega, \mathcal{F}, \Pr_{\mu})$ we can construct random variables $(Z_n)_{n \in \N_0}$ with values in $(E, \mathcal{E})$ and $(\tau_n)_{n \in \N_0}$ with values in $\extended{\R_{> 0}}$ such that 
    \begin{enumerate}[resume=conditions]
        \item It holds that $\tau_0 = 0$ and $n \mapsto \tau_n$ is strictly increasing. \label{thm:construction:jumps}
        \item $(Z_n)$ is a Markov chain on $(E, \mathcal{E})$ with transition kernel $\Pi$ (see \eqref{eq:transition_kernel}) and initial law $\mu$. \label{thm:construction:discrete}
        \item For all $n \in \N_0$ and $t \in \R_{\ge 0}$ with $\tau_n \le t < \tau_{n+1}$ it holds that $X_t = Z_n$. \label{thm:construction:discreteToContinuous}
    \end{enumerate}
\end{theorem}

We will make use of the following two lemmas about hitting times for the process that is given by \Cref{thm:construction}.
The first lemma formalizes the following intuition: if, for two sets $A, B \in \mathcal{E}$, the rate of jumping from outside of $A \cup B$ into $B$ is $0$, then the process must (almost surely) go through $A$ to ever reach $B$.

\begin{lemma} \label{lemma:separator_set}
    Let $K$ be a non-explosive jump kernel and $\mu$ be a probability distribution on $(E, \mathcal{E})$.
    Let $(\Omega, \mathcal{F}, \Pr_{\mu}, (X_t))$ be the Markov process associated with $K$ and $\mu$ as above. 
    Suppose $A, B \in \mathcal{E}$ are such that $\mu(B) = 0$ and $K(x, B) = 0$ for all $x \notin A \cup B$, and define $T_A \coloneqq \inf\{t \in \R_{\ge 0} \mid X_t \in A\}$ and $T_B \coloneqq \inf\{t \in \R_{\ge 0} \mid X_t \in B\}$. 
    It holds that $\Pr_{\mu}[T_B \le T_A, T_B < \infty] = 0$.
\end{lemma}

\begin{proof}
    Note that by \Cref{thm:construction} (\ref{thm:construction:discreteToContinuous}) it holds that
    \begin{align*}
        \Pr_{\mu}[T_B \le T_A, T_B < \infty]
        &\le \Pr_{\mu}[\exists m \in \N_0: \tau_m < \infty, Z_m \in B, \forall k \le m: Z_k \notin A \cup B] \\
        &\le \Pr_{\mu}[\exists m \in \N: Z_m \in B, Z_{m-1} \notin A \cup B] + \Pr_{\mu}[Z_0 \in B] .
    \end{align*}
    By \Cref{thm:construction} (\ref{thm:construction:discrete}) and the assumption that $\mu(B) = 0$ it holds that $\Pr_{\mu}[Z_0 \in B] = 0$ and
    \[
        \Pr_{\mu}[Z_m \in B, Z_{m-1} \notin A \cup B] = \int_{(A \cup B)^c} \int_{(A \cup B)^c} \Pi(y, B) \Pi^{m-1}(x, \diff y) \mu(\diff x).
    \]
    As we assume $K(y, B) = 0$ for $y \notin A \cup B$ it holds that inside the integral $\Pi(y, B) = 0$ and taking a union bound over $m \in \N$ concludes the proof.
\end{proof}

The second lemma states that, if we consider a sequence of regions $A_1, \dots, A_k \in \mathcal{E}$, each of which we enter with rate at most $\rho$, then the probability of hitting those sets in order before some time $t$ is upper-bounded by the probability that a Poisson random variable with rate $\rho t$ is at least $k$.
Before proving this, we show the following helpful bound on the transition family.
\begin{claim} \label{claim:transition_bound}
    Let $K$ be a jump kernel on $(E, \mathcal{E})$ with associated (sub-)Markov transition family $(P_t)$.
    Let $A \in \mathcal{E}$ be such that there is some $\rho \in \R_{\ge 0}$ such that $K(x, A)$ for $x \notin A$.
    For all $t \in \R_{\ge 0}$ and $x \notin A$ it holds that $P_t(x, A) \le \rho \cdot t$.
\end{claim}

\begin{proof}
    We show that the statement holds for each of the recursively defined sub-Markov families $(P^{(n)}_t)$ (see \eqref{eq:transition_family_base} and \eqref{eq:transition_family_step}) via induction on $n$.
    The result follows then from taking $n \to \infty$.
    For $n = 0$, the statement is trivial as $P^{(0)}_t(x, A) = \eulerE^{- \kappa(x) \cdot t} \delta_{x}(A) = 0$ for all $x \notin A$.
    Suppose the statement holds for some $n \ge 0$, we use the fact that $P_t^{(n)}$ is sub-stochastic to obtain
    \begin{align*}
        P_t^{(n+1)}(x, A) 
        &= P_t^{(0)}(x, A) + \int_0^t \eulerE^{- \kappa(x) \cdot s} \int_{E} P_{t-s}^{(n)}(x, A) K(x, \diff y) \diff s \\
        &\le \int_0^t \eulerE^{- \kappa(x) \cdot s} \cdot [\rho \cdot (t - s) \cdot K(x, A^c) + \rho] \diff s \\
        &\le \rho \int_0^t \eulerE^{- \kappa(x) \cdot s} ( 1 - \kappa(x) \cdot s) \diff s + \rho \cdot t \int_0^t \eulerE^{- \kappa(x) \cdot s} \kappa(x) \diff s \\
        &= \rho \cdot t . \qedhere
    \end{align*}
\end{proof}

We now use \Cref{claim:transition_bound} to prove the following Lemma.
\begin{lemma} \label{lemma:ordered_stopping_times}
    Let $K$ be a non-explosive jump kernel and $\mu$ be a probability distribution on $(E, \mathcal{E})$. 
    Let $(\Omega, \mathcal{F}, \Pr_{\mu}, (X_t))$ be the Markov process associated with $K$ and $\mu$ as above.
    Suppose $A_1, \dots, A_k$ are such that 
    \begin{enumerate}
        \item For all $i \in [k]$ we have $\mu(A_i) = 0$.
        \item There is some $\rho \in \R_{\ge  0}$ such that for all $i \in [k]$ and $x \notin A_i$ it holds that $K(x, A_i) \le \rho$.
    \end{enumerate}
    Set $T_i \coloneqq \inf\{t \in \R_{\ge 0} \mid X_t \in A_i\}$.
    For all $t \in \R_{\ge 0}$ it holds that
    \[
        \Pr_{\mu}[T_1 < T_2 < \dots < T_k \le t] \le 1 - F_{\rho t}(k-1),
    \]
    where $F_{\rho t}$ is the cumulative distribution function of a Poisson distribution of rate $\rho t$. 
\end{lemma}

\begin{proof}
    To simplify notation, set $T_0 \coloneqq 0$. 
    For $i \in [k]$, define
    \[
        S_i \coloneqq \begin{cases}
            T_{i} - T_{i - 1} &\text{ if } T_{i} > T_{i-1}, \\
            \infty &\text{ otherwise, }
        \end{cases}
    \]
    and $\zeta_i = \sum_{j=1}^{i} S_j$, and note that
    \[
        \Pr_{\mu}[T_1 < T_2 < \dots < T_k \le t] = \Pr_{\mu}[\zeta_k \le t] .
    \]
    The following claim is key for proving this lemma.
    \begin{claim} \label{claim:waiting_times}
        For every $s \in \R_{\ge 0}$ and every $i \in [k]$ it holds that $\Pr_{\mu}[S_i > s \mid \mathcal{F}^{X}_{T_{i-1}}] \ge \eulerE^{- \rho s}$.
    \end{claim}
     Next, let $G_{i, \rho}$ be the cumulative distribution function of an Erlang distribution with shape $i$ and rate $\rho$. 
     We will use \Cref{claim:waiting_times} to argue inductively that $\Pr_{\mu}[\zeta_i \le t] \le G_{i, \rho}(t)$.
     For $i = 1$, observe that $\zeta_1 = S_1$ and by \Cref{claim:waiting_times} we have
     \[
        \Pr_{\mu}[S_1 \le t] \le 1 - \eulerE^{-\rho t} = G_{1, \rho}(t). 
     \]
     For the induction step, we use \Cref{claim:waiting_times} and the fact that $\zeta_i$ is $\mathcal{F}^{X}_{T_i}$-measurable to obtain
     \[
        \Pr_{\mu}[\zeta_{i+1} \le t] 
        = \E[\Pr_{\mu}[S_{i+1} \le t - \zeta_{i} \mid \zeta_{i}]] 
        \le 1 - \E[\ind{\zeta_i > t} + \ind{\zeta_i \le t} \cdot \eulerE^{- \rho (t - \zeta_i)}].
     \]
     Further, by the induction hypothesis and the fact that $x \mapsto \ind{x > t} + \ind{x \le t} \cdot \eulerE^{-\rho (t - x)}$ is non-decreasing, we have
     \[
        \E[\ind{\zeta_i > t} + \ind{\zeta_i \le t} \cdot \eulerE^{- \rho (t - \zeta_i)}] \ge 1 - G_{i, \rho}(t) + \eulerE^{-\rho t} \int_{0}^{t} \eulerE^{\rho x} G_{i, \rho}(\diff x) = 1 - G_{i+1, \rho}(t)  
     \]
     Hence, we have $\Pr_{\mu}[\zeta_{i+1} \le t] \le G_{i+1, \rho}(t)$ which concludes the induction.
     Finally, observing that $G_{k, \rho}(t) = 1 - F_{\rho t}(k - 1)$ concludes the proof.
\end{proof}

\begin{proof}[Proof of \Cref{claim:waiting_times}]
    Note that
    \[
        \Pr_{\mu}[S_i > s \mid \mathcal{F}^{X}_{T_{i-1}}] 
        = \ind{T_i \le T_{i-1}} + \Pr_{\mu}[T_{i} > T_{i-1} + s \mid \mathcal{F}^{X}_{T_{i-1}}].
    \]
    Hence, our main concern is to lower-bound $\Pr_{\mu}[T_{i} > T_{i-1} + s \mid \mathcal{F}^{X}_{T_{i-1}}]$.
    To this end, for every $n \in \N$, we consider the event $D^{(s)}_n \coloneqq \{T_i > T_{i-1}\} \cap \{\forall k \in \{0, \dots, n-1\}: X_{T_{i-1} + \frac{k s}{n}} \notin A_i\}$.
    By \Cref{thm:construction} (\ref{thm:construction:sample_paths}), we know that $(X_t)$ has piecewise constant sample paths. 
    Thus, $\ind{D^{(s)}_n} \to \ind{T_i > T_{i-1} + s}$ pointwise as $n \to \infty$, and $\lim_{n \to \infty} \Pr_{\mu}[D^{(s)}_{n} \mid \mathcal{F}^{X}_{T_{i-1}}] = \Pr_{\mu}[T_{i} > T_{i-1} + s \mid \mathcal{F}^{X}_{T_{i-1}}]$ by dominated convergence.
    Further, using the strong Markov property from \Cref{thm:construction} (\ref{thm:construction:strongMarkov}) and \Cref{claim:transition_bound}, we obtain
    \begin{align*}
        \Pr_{\mu}[D^{(s)}_{n} \mid \mathcal{F}^{X}_{T_{i-1}}]
        &= \E[\ind{D^{(s - s/n)}_{n-1}} \cdot \Pr_{\mu}[X_{T_{i-1} + s} \notin A_i \mid \mathcal{F}^{X}_{T_{i-1} + \frac{n-1}{n}s}] \mid \mathcal{F}^{X}_{T_{i-1}}] \\
        &= \E[\ind{D^{(s - s/n)}_{n-1}} \cdot P_{s/n}(X_{T_{i-1} + \frac{n-1}{n}s}, A_i^c) \mid \mathcal{F}^{X}_{T_{i-1}}] \\
        &\ge \Pr_{\mu}[D^{(s - s/n)}_{n-1} \mid \mathcal{F}^{X}_{T_{i-1}}] \cdot \left(1 - \frac{\rho s}{n}\right) .
    \end{align*}
    Applying this argument recursively and taking the limit $n \to \infty$ yields
    \[
        \Pr_{\mu}[T_{i} > T_{i-1} + s \mid \mathcal{F}^{X}_{T_{i-1}}] \ge \ind{T_i > T_{i - 1}} \cdot \lim_{n \to \infty} \left(1 - \frac{\rho s}{n}\right)^n = \ind{T_i > T_{i - 1}} \cdot \eulerE^{- \rho s}.
    \]
    Hence, we have
    \[
        \Pr_{\mu}[S_i > s \mid \mathcal{F}^{X}_{T_{i-1}}] 
        \ge \ind{T_i \le T_{i-1}} + \ind{T_i > T_{i - 1}} \cdot \eulerE^{- \rho s}
        \ge \eulerE^{- \rho s} . \qedhere
    \]
\end{proof}

\section{Constructing spatial birth-death dynamics for Gibbs point processes} \label{sec:jump_process_for_gpp}
In this section we construct the jump process that we will use to prove uniqueness of infinite volume Gibbs measures.
To this end, fix some activity $\activity \in \R_{\ge 0}$ and let $\potential$ be a pair potential that is locally stable with constant $\localStability \in \R_{\ge 0}$.
For every $\xi \in \feasibleCountingMeasures$ and every $\subregion \in \boundedBorel$ we define  $\jumpKernel_{\subregion \mid \xi}: \countingMeasures_{\subregion \mid \xi} \times \countingAlgebra_{\subregion \mid \xi} \to \R_{\ge 0}$ via
\begin{align} \label{eq:jump_kernel}
    \jumpKernel_{\subregion \mid \xi}(\eta, A) \coloneqq\int_{\subregion} \ind{A}(\eta \setminus \delta_x) \eta(\diff x) + \activity \int_{\subregion} \ind{A}(\eta + \delta_x) \cdot \eulerE^{-\influence(x, \eta + \xi_{\subregion^c})} \volume(\diff x).
\end{align}

\begin{lemma} \label{lemma:jump_kernel_non_exp}
    Let $\potential$ be locally stable, $\subregion \in \boundedBorel$ and $\xi \in \feasibleCountingMeasures$.
    Then $\jumpKernel_{\subregion \mid \xi}$, defined as in \eqref{eq:jump_kernel}, is a non-explosive jump kernel on $(\countingMeasures_{\subregion \mid \xi}, \countingAlgebra_{\subregion \mid \xi})$.
\end{lemma}

\begin{proof}
    Using the lower bound on $\influence$ from \Cref{lemma:influence_bound} and the fact that $\eta + \xi_{\subregion^c} \in \feasibleCountingMeasures$ for all $\eta \in \countingMeasures_{\subregion \mid \xi}$, we have
    \[
        \jumpKernel_{\subregion \mid \xi}(\eta, \countingMeasures_{\subregion}) \le \eta(\subregion) + \activity \eulerE^{\localStability} \cdot \volume(\subregion)
        < \infty,
    \]
    proving that $\jumpKernel_{\subregion \mid \xi}$ is indeed finite.

    To prove that $\jumpKernel_{\subregion \mid \xi}$ is non-explosive, we are going to use \Cref{lemma:nonexp_test} with 
    \[
        E_n = \{\eta \in \countingMeasures_{\subregion \mid \xi} \mid \eta(\subregion) \le n\}, 
    \]
    and with $f: \countingMeasures_{\subregion \mid \xi} \to \R_{\ge 0}, \eta \mapsto \eta(\subregion) + 1$.
    We check that this choice of $(E_n)$ and $f$ satisfies the requirements of \Cref{lemma:nonexp_test}.
    Obviously, we have that $E_n$ converges monotonically to $\countingMeasures_{\subregion \mid \xi}$.
    Further, by \Cref{lemma:influence_bound} and the fact that $\eta + \xi_{\subregion^c} \in \feasibleCountingMeasures$ for all $\eta \in \countingMeasures_{\subregion \mid \xi}$, it holds that for every $n \in \N$
    \[
        \sup_{\eta \in E_n} \jumpKernel_{\subregion \mid \xi}(\eta, \countingMeasures_{\subregion \mid \xi}) \le n + \activity \eulerE^{\localStability} \cdot \volume(\subregion)
        < \infty.
    \]
    Next, observe that $\inf_{\eta \in \countingMeasures_{\subregion \mid \xi} \setminus E_n} f(\eta) =  n + 2$ and therefore $\lim_{n \to \infty} \inf_{\eta \in \countingMeasures_{\subregion \mid \xi} \setminus E_n} f(\eta) = \infty$.
    Finally, we have for all $\eta \in \countingMeasures_{\subregion \mid \xi}$ that
    \begin{align*}
        (\jumpKernel_{\subregion \mid \xi} f)(\eta)
        &\le \eta(\subregion)^2 + (\eta(\subregion) + 2) \cdot \activity  \eulerE^{L} \cdot \volume(\subregion) \\
        &\le (\eta(\subregion) + 2 \activity  \eulerE^{L} \cdot \volume(\subregion)) \cdot f(\eta) \\
        &\le (\jumpKernel_{\subregion \mid \xi}(\eta, \countingMeasures_{\subregion \mid \xi}) + 2 \activity  \eulerE^{L} \cdot \volume(\subregion)) \cdot f(\eta).
    \end{align*}
    Hence, setting $c = 2 \activity  \eulerE^{L} \cdot \volume(\subregion)$ and applying \Cref{lemma:nonexp_test} shows that $\jumpKernel_{\subregion \mid \xi}$ is non-explosive and concludes the proof.
\end{proof}

Throughout the following subsections, we study two important properties of the transition family $(P_t)$ that is associated with a jump kernel $\jumpKernel_{\subregion \mid \xi}$.
Firstly, we derive conditions such that $P_t(\eta, \cdot)$ converges rapidly to $\gibbs_{\subregion \mid \xi}$ (viewed as a distribution on $(\countingMeasures_{\subregion \mid \xi}, \countingAlgebra_{\subregion \mid \xi})$) in total variation distance as $t \to \infty$.
Secondly, we give conditions such that, for two transitions families $(P^{\xi}_t)$ and $(P^{\zeta}_t)$ associated with $\jumpKernel_{\subregion \mid \xi}$ and $\jumpKernel_{\subregion \mid \zeta}$ for boundary configurations $\xi, \zeta \in \feasibleCountingMeasures$, it holds that the total variation distance between $P^{\xi}_t(\eta, \cdot)$ and $P^{\zeta}_t(\eta, \cdot)$ when projected on some subregion $\subregion' \subset \subregion$ that is sufficiently far from $\subregion^c$ only increases slowly as a function of $t$.
Combining both properties with \Cref{lemma:uniqueness_from_ssm} will yield the desired uniqueness result for infinite-volume Gibbs measures.

\subsection{Identity coupling}
We start by constructing a coupling for the jump kernels $\jumpKernel_{\subregion \mid \xi}, \jumpKernel_{\subregion \mid \zeta}$ for any $\subregion \in \boundedBorel$ and $\xi, \zeta \in \feasibleCountingMeasures$, which we will call the \emph{identity coupling}.
We remark that a similar coupling for spatial birth-death processes has been studied before by Schuhmacher and Stucki \cite{schuhmacher2014gibbs} for applying Stein's method to Gibbs point processes.

\begin{lemma} \label{lemma:identity_coupling}
    Let $\subregion \in \boundedBorel$ and $\xi, \zeta \in \feasibleCountingMeasures$, and for every $\eta_1 \in \countingMeasures_{\subregion \mid \xi}$, $ \eta_2 \in \countingMeasures_{\subregion \mid \zeta}$ and $A \in \countingAlgebra_{\subregion \mid \xi} \otimes \countingAlgebra_{\subregion \mid \zeta}$ define
    \begin{align*}
        D_1((\eta_1, \eta_2), A) &\coloneqq \int_{\subregion} \ind{A}(\eta_1 \setminus \delta_x, \eta_2) (\eta_1 \setminus \eta_2)(\diff x) \\
        D_2((\eta_1, \eta_2), A) &\coloneqq \int_{\subregion} \ind{A}(\eta_1, \eta_2  \setminus \delta_x) (\eta_2 \setminus \eta_1)(\diff x) \\
        D_{\cap}((\eta_1, \eta_2), A) &\coloneqq \int_{\subregion} \ind{A}(\eta_1 \setminus \delta_x, \eta_2 \setminus \delta_x) (\eta_1 \cap \eta_2)(\diff x) \\
        B_{1}((\eta_1, \eta_2), A) &\coloneqq \activity \int_{\subregion} \max\left(0, \eulerE^{-\influence(x, \eta_1 + \xi_{\subregion^c})} - \eulerE^{-\influence(x, \eta_2 + \zeta_{\subregion^c})}\right) \cdot \ind{A}(\eta_1 + \delta_x, \eta_2)\volume(\diff x) \\
        B_{2}((\eta_1, \eta_2), A) &\coloneqq \activity \int_{\subregion} \max\left(0, \eulerE^{-\influence(x, \eta_2 + \zeta_{\subregion^c})} - \eulerE^{-\influence(x, \eta_1 + \xi_{\subregion^c})}\right) \cdot \ind{A}(\eta_1, \eta_2 + \delta_x)\volume(\diff x) \\
        B_{\cap}((\eta_1, \eta_2), A) &\coloneqq \activity \int_{\subregion} \min\left(\eulerE^{-\influence(x, \eta_1 + \xi_{\subregion^c})}, \eulerE^{-\influence(x, \eta_2 + \zeta_{\subregion^c})}\right) \cdot \ind{A}(\eta_1 + \delta_x, \eta_2 + \delta_x)\volume(\diff x) .
    \end{align*}
    Then $\jumpKernel \coloneqq D_1 + D_2 + D_{\cap} + B_1 + B_2 + B_{\cap}$ is a coupling of $\jumpKernel_{\subregion \mid \xi}$ and $\jumpKernel_{\subregion \mid \zeta}$ in the sense of \Cref{def:coupling_kernel}.
\end{lemma}

\begin{remark}
    We should think of the way that \Cref{lemma:identity_coupling} presents the construction of the coupled jump kernel $\jumpKernel$ in terms of its associated Markov process $(X_t^{(1)}, X_t^{(2)})_{t \in \R_{\ge 0}}$ on $\countingMeasures_{\subregion \mid \xi} \times \countingMeasures_{\subregion \mid \zeta}$ as follows:
    The kernels $D_1$ and $D_2$ represent the rate at which only $X_t^{(1)}$ or only $X_t^{(2)}$ remove a point (i.e., $D_i$ is the death rate that is exclusive to $X_t^{(i)}$).
    In contrast, $D_{\cap}$ is the rate at which both processes remove a point that they have in common (i.e., $D_{\cap}$ is the shared death rate on $X_t^{(1)} \cap X_t^{(2)}$).
    Analogously, $B_1$ and $B_2$ represent the rate at which a point is added to only $X_t^{(1)}$ or only $X_t^{(2)}$ (i.e., birth rate exclusive to one process), while $B_{\cap}$ is the rate at which both processes add the same point (i.e., shared birth rate).
\end{remark}

\begin{proof}[Proof of \Cref{lemma:identity_coupling}]
    Checking that $K$ is indeed a jump kernels can be done similarly to the proof of \Cref{lemma:jump_kernel_non_exp}.
    To show that $K$ is a coupling of $\jumpKernel_{\subregion \mid \xi}$ and $\jumpKernel_{\subregion \mid \zeta}$,
    let $f: \countingMeasures_{\subregion \mid \xi} \to \R_{\ge 0}$ be bounded and measurable, and observe that 
    \[
        \int_{\countingMeasures_{\subregion \mid \xi} \times \countingMeasures_{\subregion \mid \zeta}} (f(\tau_1) - f(\eta_1)) D_2((\eta_1, \eta_2), \diff \tau_1 \times \diff \tau_2) 
        = \int_{\subregion} (f(\eta_1) - f(\eta_1)) (\eta_2 \setminus \eta_1)(\diff x) = 0
    \]
    and analogously
    \[
        \int_{\countingMeasures_{\subregion \mid \xi} \times \countingMeasures_{\subregion \mid \zeta}} (f(\tau_1) - f(\eta_1)) B_2((\eta_1, \eta_2), \diff \tau_1 \times \diff \tau_2) = 0 .
    \]    
    Further, using the fact that $\eta_1 = (\eta_1 \setminus \eta_2) + (\eta_1 \cap \eta_2)$ we have
    \[
        \int_{\countingMeasures_{\subregion \mid \xi} \times \countingMeasures_{\subregion \mid \zeta}} (f(\tau_1) - f(\eta_1)) (D_1 + D_{\cap})((\eta_1, \eta_2), \diff \tau_1 \times \diff \tau_2)
        = \int_{\subregion} (f(\eta_1 \setminus \delta_x) - f(\eta_1)) \eta_1(\diff x),
    \]
    and, since $a = \max(0, a-b) + \min(a, b)$ for all $a, b \in \R_{\ge 0}$,\footnote{In fact, the integral on the right-hand side should only be over $\{x \in \subregion \mid \eta + \delta_x \in \countingMeasures_{\subregion \mid \xi}\}$. However, since $\eta \in \countingMeasures_{\subregion \mid \xi}$, it can be shown that all $x \in \subregion$ with $\eta + \delta_x \notin \countingMeasures_{\subregion \mid \xi}$ satisfy $\influence(x, \eta + \xi_{\subregion^c}) = \infty$, justifying this expression.}
     \[
        \int_{\countingMeasures_{\subregion \mid \xi} \times \countingMeasures_{\subregion \mid \zeta}} (f(\tau_1) - f(\eta_1)) (B_1 + B_{\cap})((\eta_1, \eta_2), \diff \tau_1 \times \diff \tau_2)
        = \activity \int_{\subregion} \eulerE^{-\influence(x, \eta_1 + \xi_{\subregion^c})} \cdot (f(\eta_1 + \delta_x) - f(\eta_1)) \volume(\diff x).
    \]
    Hence, it holds that
    \begin{align*}
        \int_{\countingMeasures_{\subregion \mid \xi} \times \countingMeasures_{\subregion \mid \zeta}} (f(\tau_1) - f(\eta_1)) K((\eta_1, \eta_2), \diff \tau_1 \times \diff \tau_2)
        &= \int_{\countingMeasures_{\subregion \mid \xi}} (f(\tau_1) - f(\eta_1)) \jumpKernel_{\subregion \mid \xi}(\eta_1, \diff \tau_1) \\
        &= (\jumpKernel_{\subregion \mid \xi} f)(\eta_1) - \jumpKernel_{\subregion \mid \xi}(\eta_1, \countingMeasures_{\subregion \mid \xi}) \cdot f(\eta_1) .
    \end{align*}
    By symmetry, it then follows that $K$ is indeed a coupling in the sense of \Cref{def:coupling_kernel}.    
\end{proof}

\subsection{Rapid mixing} \label{subsec:mixing}
We proceed by showing that, for sufficiently small activity $\activity$, the Markov transition family associated to $\jumpKernel_{\subregion \mid \xi}$ converges rapidly to $\gibbs_{\subregion \mid \xi}$ (noting that $\gibbs_{\subregion \mid \xi}$ is well-defined by \Cref{lemma:bound_partition_function}).
We start by identifying the finite volume Gibbs measure $\gibbs_{\subregion \mid \xi}$ as a stationary distribution of the Markov transition family associated to $\jumpKernel_{\subregion \mid \xi}$.
More precisely, recall that by \Cref{lemma:gibbs_support} it holds that $\gibbs_{\subregion \mid \xi}$ is only supported on $\countingMeasures_{\subregion \mid \xi}$.
Thus, we can uniquely identify $\gibbs_{\subregion \mid \xi}$ with a distribution $\mu' \in \probMeasures(\countingMeasures_{\subregion \mid \xi}, \countingAlgebra_{\subregion \mid \xi})$ with $\mu'(A) = \gibbs_{\subregion \mid \xi}(A)$ for all $A \in \countingAlgebra_{\subregion \mid \xi}$.
With some abuse of terminology and notation, we may not differentiate between $\gibbs_{\subregion \mid \xi}$ and the associated distribution on $(\countingMeasures_{\subregion \mid \xi}, \countingAlgebra_{\subregion \mid \xi})$ in the following statements.

\begin{lemma} \label{lemma:jump_kernel_reversible}
    Let $\potential$ be locally stable, $\subregion \in \boundedBorel$, $\xi \in \feasibleCountingMeasures$ and let $\jumpKernel_{\subregion \mid \xi}$ be defined as in \eqref{eq:jump_kernel}.
    The Markov transition family associated to $\jumpKernel_{\subregion \mid \xi}$ is reversible with respect to $\gibbs_{\subregion \mid \xi}$.
\end{lemma}
\begin{proof}
    First, note that by \Cref{lemma:bound_partition_function} it holds that $\gibbs_{\subregion \mid \xi}$ is well-defined.
    Using \Cref{lemma:reversible_kernel} we only need to check that
    \[
        \int_{A} \jumpKernel_{\subregion \mid \xi}(\eta, B) \gibbs_{\subregion \mid \xi}(\diff \eta) = \int_{B} \jumpKernel_{\subregion \mid \xi}(\eta, A) \gibbs_{\subregion \mid \xi}(\diff \eta) 
    \]
    for all $A, B \in \countingAlgebra_{\subregion \mid \xi}$.
    Writing both sides explicitly shows that it suffices to show that 
    \[
        \int_{A} \int_{\subregion} \ind{B}(\eta \setminus \delta_x) \eta(\diff x) \gibbs_{\subregion \mid \xi}(\diff \eta)
        = \int_{B} \activity \int_{\subregion} \eulerE^{- \influence(x, \eta + \xi_{\subregion^c})} \ind{A}(\eta + \delta_x) \eta(\diff x) \gibbs_{\subregion \mid \xi}(\diff \eta) ,
    \]
    which follows immediately from applying the GNZ equation from \Cref{lemma:gnz} to the function $(x, \eta) \mapsto \ind{A}(\eta) \cdot \ind{B}(\eta \setminus \delta_x)$.
\end{proof}

Next, we bound the speed of convergence by using \Cref{lemma:mixing} and the identity coupling from \Cref{lemma:identity_coupling}.

\begin{lemma} \label{lemma:jump_process_mixing}
    Let $\potential$ be locally stable with constant $\localStability$ and weakly tempered with constant $\weakTempered_{\potential}$.
    For all $\activity < \frac{1}{\weakTempered_{\potential} \eulerE^{\localStability}}$ there is some $\delta > 0$ such that for all $\subregion \in \boundedBorel$ and $\xi \in \feasibleCountingMeasures$ the following holds:
    For $\jumpKernel_{\subregion \mid \xi}$ as in \eqref{eq:jump_kernel}, the associated Markov transition family $(P_t)$ satisfies
    \[
        \dtv{P_t(\eta, \cdot)}{\gibbs_{\subregion \mid \xi}} \le \eulerE^{- \delta t} \cdot (\activity \eulerE^{\localStability} \volume(\subregion) + \eta(\subregion))
    \]
    for all $\eta \in \countingMeasures_{\subregion \mid \xi}$ and all $t \in \R_{\ge 0}$.
    In particular, $\delta$ may depend on $\weakTempered_{\potential}$, $\localStability$ and $\activity$, but is independent of $\subregion$, $\eta$, $\xi$ and $t$.
\end{lemma}

\begin{proof}
    We aim to prove this statement using \Cref{lemma:mixing}.
    By \Cref{lemma:jump_kernel_non_exp} we know that $\jumpKernel_{\subregion \mid \xi}$ is non-explosive, and \Cref{lemma:jump_kernel_reversible} implies that $\gibbs_{\subregion \mid \xi}$ is a stationary distribution of the associated Markov transition family $(P_t)$.
    
    Let $\jumpKernel$ be the identity coupling of $\jumpKernel_{\subregion \mid \xi}$ with itself as provided by \Cref{lemma:identity_coupling} (i.e., the coupling given by \Cref{lemma:identity_coupling} for $\zeta = \xi$), and let $D_1, D_2, D_{\cap}, B_1, B_2, B_{\cap}$ be defined accordingly.
    To conclude our proof, it remains to find a measurable function $f: \countingMeasures_{\subregion \mid \xi} \times \countingMeasures_{\subregion \mid \xi} \to \R_{\ge 0}$ such that the conditions (\ref{item:mixing:threshold}) and (\ref{item:mixing:contraction}) of \Cref{lemma:mixing} are satisfied for suitable constants $c, \delta > 0$.
    To this end, we set $f(\eta_1, \eta_2) \coloneqq (\eta_1 \symdiff \eta_2)(\subregion)$.
    Note that $f$ can equivalently be expressed as $\sup_{\Delta \in \borel_{\subregion}} \absolute{\eta_1(\Delta) - \eta_2(\Delta)}$, showing that it is indeed measurable.
    Further, it holds that $f(\eta_1, \eta_2) \ge 1$ if and only if $\eta_1 \neq \eta_2$.
    Thus, condition (\ref{item:mixing:threshold}) of \Cref{lemma:mixing} is satisfied by $f$ for $c = 1$.  
    
    To verify condition (\ref{item:mixing:contraction}) of \Cref{lemma:mixing}, we first observe the following:
    \begin{itemize}
        \item For all $x \in \support_{\eta_1 \setminus \eta_2}$ we have $f(\eta_1 \setminus \delta_x, \eta_2) = f(\eta_1, \eta_2) - 1$. Hence, it holds that
        \[
            (D_1 f)(\eta_1, \eta_2) = (f(\eta_1, \eta_2) - 1) \cdot (\eta_1 \setminus \eta_2)(\subregion).
        \]
        \item For all $x \in \support_{\eta_2 \setminus \eta_1}$ we have $f(\eta_1, \eta_2  \setminus \delta_x) = f(\eta_1, \eta_2) - 1$. Hence, it holds that
        \[
            (D_2 f)(\eta_1, \eta_2) = (f(\eta_1, \eta_2) - 1) \cdot (\eta_2 \setminus \eta_1)(\subregion).
        \]
        \item For all $x \in \support_{\eta_2 \cap \eta_1}$ we have $f(\eta_1 \setminus \delta_x, \eta_2  \setminus \delta_x) = f(\eta_1, \eta_2)$.  Hence, it holds that
        \[
            (D_{\cap} f)(\eta_1, \eta_2) = f(\eta_1, \eta_2) \cdot (\eta_1 \cap \eta_2)(\subregion).
        \]
        \item For all $x \in \subregion$ we have $f(\eta_1 + \delta_x, \eta_2) =  f(\eta_1, \eta_2) + 1$, $f(\eta_1, \eta_2 + \delta_x) = f(\eta_1, \eta_2) + 1$ and $f(\eta_1 + \delta_x, \eta_2 + \delta_x) = f(\eta_1, \eta_2)$. Hence, it holds that
        \begin{align*}
            (B_1 f)(\eta_1, \eta_2) &= (f(\eta_1, \eta_2) + 1) \cdot  \activity \int_{\subregion} \max\left(0, \eulerE^{-\influence(x, \eta_1 + \xi_{\subregion^c})} - \eulerE^{-\influence(x, \eta_2 + \xi_{\subregion^c})}\right) \volume(\diff x) , \\  
            (B_2 f)(\eta_1, \eta_2) &= (f(\eta_1, \eta_2) + 1) \cdot  \activity \int_{\subregion} \max\left(0, \eulerE^{-\influence(x, \eta_2 + \xi_{\subregion^c})} - \eulerE^{-\influence(x, \eta_1 + \xi_{\subregion^c})}\right) \volume(\diff x) , \\
            (B_{\cap} f)(\eta_1, \eta_2) &= f(\eta_1, \eta_2) \cdot  \activity \int_{\subregion} \min\left(\eulerE^{-\influence(x, \eta_1 + \xi_{\subregion^c})}, \eulerE^{-\influence(x, \eta_2 + \xi_{\subregion^c})}\right) \volume(\diff x) . 
        \end{align*}
    \end{itemize}
    Further, since $\eta_1 \cup \eta_2 = \eta_1 \symdiff \eta_2 + \eta_1 \cap \eta_2$ and $\max(0, a - b) + \max(0, b - a) + \min(a, b) = \max(a, b)$ for every $a, b \in \R$, we can write the rate function of $K$ as 
    \[
        \kappa(\eta_1, \eta_2)
        \coloneqq K((\eta_1, \eta_2), \countingMeasures_{\subregion \mid \xi} \times \countingMeasures_{\subregion \mid \xi})
        = (\eta_1 \cup \eta_2)(\subregion) + \activity \int_{\subregion} \max\left(\eulerE^{-\influence(x, \eta_1 + \xi_{\subregion^c})}, \eulerE^{-\influence(x, \eta_2 + \xi_{\subregion^c})}\right) \volume(\diff x) .
    \]
    Combining these observations yields
    \[
        (Kf)(\eta_1, \eta_2) = \kappa(\eta_1, \eta_2) \cdot f(\eta_1, \eta_2) - f(\eta_1, \eta_2) + \activity \int_{\subregion} \absolute{\eulerE^{-\influence(x, \eta_1 + \xi_{\subregion^c})} - \eulerE^{-\influence(x, \eta_2 + \xi_{\subregion^c})}} \volume(\diff x).
    \]
    Next, we use the following claim that we prove below.
    \begin{claim} \label{claim:drift}
    	There exists some $\delta > 0$, independent of $\subregion$, such that, for all $\eta_1, \eta_2 \in \countingMeasures_{\subregion}$, it holds that
    	\[
    		\activity \int_{\subregion} \absolute{\eulerE^{-\influence(x, \eta_1 + \xi_{\subregion^c})} - \eulerE^{-\influence(x, \eta_2 + \xi_{\subregion^c})}} \volume(\diff x) 
    		\le (1 - \delta) \cdot (\eta_1 \symdiff \eta_2)(\subregion) .
    	\]
    \end{claim}
    Provided \Cref{claim:drift}, we have
    \[
    	(Kf)(\eta_1, \eta_2) \le \left(\kappa(\eta_1, \eta_2) - \delta\right) \cdot f(\eta_1, \eta_2)
    \]
    for some $\delta > 0$ independent of $\subregion$ and $\eta_1, \eta_2$.
    Thus, applying \Cref{lemma:mixing} yields
    \[
    	\dtv{P_t(\eta, \cdot)}{\gibbs_{\subregion \mid \xi}} \le \eulerE^{- \delta \cdot t} \int_{\countingMeasures_{\subregion \mid \xi}} f(\eta, \zeta) \gibbs_{\subregion \mid \xi}(\diff \zeta) 
    \]
    for $\delta > 0$ as desired.
    Finally, observe that $f(\eta, \zeta) = (\eta \symdiff \zeta)(\subregion) \le \eta(\subregion) + \zeta(\subregion)$.
    Hence, applying the GNZ equation from \Cref{lemma:gnz} with the function $(x, \zeta) \mapsto \ind{\countingMeasures_{\subregion \mid \xi}} (\zeta)$ and using the bound on the influence from \Cref{lemma:influence_bound} yields
    \[
        \int_{\countingMeasures_{\subregion \mid \xi}} f(\eta, \zeta) \gibbs_{\subregion \mid \xi}(\diff \zeta) 
        \le \eta(\subregion) + \activity \eulerE^{\localStability} \volume(\subregion),
    \]
    which concludes the proof.
\end{proof}

\begin{proof}[Proof of \Cref{claim:drift}]
	Let $k = (\eta_1 \symdiff \eta_2)(\subregion)$.
	We start by constructing a sequence $\zeta_0, \dots, \zeta_k \in \countingMeasures_{\subregion \mid \xi}$ such that $\zeta_0 = \eta_1$, $\zeta_k = \eta_2$ and for all $1 \le i \le k$ it holds that $\zeta_{i - 1} \symdiff \zeta_{i} = \delta_z$ for some $z \in \subregion$.
	To this end, let $k_1 = (\eta_1 \setminus \eta_2)(\subregion)$ and $k_2 = (\eta_2 \setminus \eta_1)(\subregion)$, and note that $k = k_1 + k_2$.
	Let  $\eta_1 \setminus \eta_2 = \sum_{j = 1}^{k_1} \delta_{y_j}$ and let $\eta_2 \setminus \eta_1 = \sum_{j = k_1 + 1}^{k_1 + k_2} \delta_{y_j}$ for some sequence $y_1, \dots, y_{k_1 + k_2} \in \subregion$.
	Now, for $0 \le i \le k_1$ set $\zeta_i = (\eta_1 \cap \eta_2) + \sum_{j = 1}^{k_1 - i} \delta_{y_j}$, and for $k_1 + 1 \le i \le k_1 + k_2$ set $\zeta_i = (\eta_1 \cap \eta_2) + \sum_{j = k_1 + 1}^{i} \delta_{y_j}$.
    By construction, $\zeta_{i}$ can be obtained from $\zeta_{i-1}$ by adding or subtracting a Dirac measure.
	Using the fact that $\zeta_{k_1} = \eta_1 \cap \eta_2$ and that $\eta_1 = (\eta_1 \cap \eta_2) + (\eta_1 \setminus \eta_2)$ and $\eta_2 = (\eta_1 \cap \eta_2) + (\eta_2 \setminus \eta_1)$, we further see that $\zeta_0 = \eta_1$ and $\zeta_{k} = \eta_2$. 
    Moreover, note that $\zeta_i \le \eta_1$ or $\zeta_i \le \eta_2$, hence $\zeta_i \in \countingMeasures_{\subregion \mid \xi}$, and the sequence $\zeta_0, \dots, \zeta_k$ satisfies the desired properties.

	By the triangle inequality, we now have
	\[
		\activity \int_{\subregion} \absolute{\eulerE^{-\influence(x, \eta_1 + \xi_{\subregion^c})} - \eulerE^{-\influence(x, \eta_2 + \xi_{\subregion^c})}} \volume(\diff x) 
		\le \sum_{i = 1}^{k} \activity \int_{\subregion} \absolute{\eulerE^{-\influence(x, \zeta_{i} + \xi_{\subregion^c})} - \eulerE^{-\influence(x, \zeta_{i-1} + \xi_{\subregion^c})}} \volume(\diff x) . 
	\] 
	Therefore, it is sufficient to show that
	\[
		\activity \int_{\subregion} \absolute{\eulerE^{-\influence(x, \zeta_{i} + \xi_{\subregion^c})} - \eulerE^{-\influence(x, \zeta_{i-1} + \xi_{\subregion^c})}} \volume(\diff x) \le 1 - \delta
	\]
	for every $1 \le i \le k$.
	By symmetry, we may assume without loss of generality that $\zeta_{i} = \zeta_{i-1} + \delta_z$ for some $z \in \subregion$.
	It then holds that 
	\[
		\influence(x, \zeta_{i} + \xi_{\subregion^c}) = \influence(x, \zeta_{i-1} + \xi_{\subregion^c}) + \potential(x, z) .
	\]
	Thus, by local stability and \Cref{lemma:influence_bound} we have
    \begin{align*}
         \absolute{\eulerE^{-\influence(x, \zeta_{i} + \xi_{\subregion^c})} - \eulerE^{-\influence(x, \zeta_{i-1} + \xi_{\subregion^c})}}
         &\le \min\left(\eulerE^{-\influence(x, \zeta_{i-1} + \xi_{\subregion^c})} \cdot \absolute{\eulerE^{-\potential(x, z)} - 1}, \eulerE^{-\influence(x, \zeta_{i} + \xi_{\subregion^c})} \cdot \absolute{1 - \eulerE^{\potential(x, z)}}\right) \\ 
         &\le \eulerE^{\localStability} \cdot \min\left(\absolute{\eulerE^{-\potential(x, z)} - 1}, \absolute{1 - \eulerE^{\potential(x, z)}}\right) \\
         &\le \eulerE^{\localStability} \cdot \left(1 - \eulerE^{-\absolute{\potential(x, z)}}\right) .
    \end{align*}
    Combining this with weak temperedness, we obtain
	\[
		\activity \int_{\subregion} \absolute{\eulerE^{-\influence(x, \zeta_{i} + \xi_{\subregion^c})} - \eulerE^{-\influence(x, \zeta_{i-1} + \xi_{\subregion^c})}} \volume(\diff x) 
		\le \activity \eulerE^{\localStability} \int_{\subregion}  1 - \eulerE^{- \absolute{\potential(x, z)}} \volume(\diff x)
		\le \activity \eulerE^{\localStability} \weakTempered_{\potential}.
	\]
	Since we assume $\activity < \frac{1}{\eulerE^{\localStability} \weakTempered_{\potential}}$, this proves the claim.
\end{proof}

\subsection{Slow percolation of disagreements} \label{sec:percolation}
Throughout this section, we will use slightly stronger assumptions than in \Cref{subsec:mixing}, namely:
\begin{itemize}
    \item We assume that $\potential$ is locally stable with constant $\localStability$ and has range at most $\range$. We may assume $\range > 0$ to avoid trivialities.
    \item We assume that the underlying space $(\groundspace, \dist)$ is the $\dimension$-dimensional Euclidean space for some fixed dimension $\dimension \in \N$, and that $\volume$ is the respective Lebesgue measure.
\end{itemize}

Moreover, we introduce the following notational conventions.
Firstly, for any two point processes $P, Q \in \probMeasures(\countingMeasures, \countingAlgebra)$ and some region $\subregion \in \boundedBorel$ we write $\projDtv{P}{Q}{\subregion}$ for the projected total variation distance $\dtv{P[\subregion]}{Q[\subregion]}$. 
Further, for every tuple $\tuple{k} = (k_1, \dots, k_{\dimension}) \in \Z^{\dimension}$ we set 
\[
    \subregion_{\tuple{k}} \coloneqq \left[(k_1 - \frac{1}{2}) \range, (k_1 + \frac{1}{2}) \range\right] \times \dots \times \left[(k_{\dimension} - \frac{1}{2}) \range, (k_{\dimension} + \frac{1}{2}) \range\right].
\]
For $n \in \N$ we set $\vertices_n \coloneqq (\Z \cap [-n, n])^{\dimension}$ and 
\[
    \subregion^{(n)} \coloneqq \bigcup_{\tuple{k} \in \vertices_n} \subregion_{\tuple{k}} = [-n - 1/2, n + 1/2]^{\dimension} .
\]
For every $\xi \in \feasibleCountingMeasures$, we simplify notation and write $\jumpKernel_{n \mid \xi}$ for $\jumpKernel_{\subregion^{(n)} \mid \xi}$ as given in \eqref{eq:jump_kernel}, $\countingMeasures_{n \mid \xi}$ for $\countingMeasures_{\subregion^{(n)} \mid \xi}$ and $\countingAlgebra_{n \mid \xi}$ for $\countingAlgebra_{\subregion^{(n)} \mid \xi}$.

Throughout the section, it will be useful to consider an undirected graph $\graph_n$ on the vertex set $\vertices_n$ with edges $\edges_n \coloneqq \{\{\tuple{j}, \tuple{k}\} \in \binom{\vertices_n}{2} \mid \norm{\tuple{j} - \tuple{k}}{\infty} = 1 \}$.
In accordance with this graph construction, we denote by $\neighborhood_n(\tuple{k}) \coloneqq \{\tuple{j} \in \vertices_n \mid \{\tuple{k}, \tuple{j}\} \in \edges_n\}$ the neighborhood of $\tuple{k} \in \vertices_n$ in $\graph_n$.
Further, we partition the vertex set $\vertices_n$ into \emph{inner vertices} $\innerVertices_n \coloneqq \{\tuple{k} \in \vertices_n \mid \norm{\tuple{k}}{\infty} < n\}$ and \emph{outer vertices} $\outerVertices_n \coloneqq \vertices_n \setminus \innerVertices_n = \{\tuple{k} \in \vertices_n \mid \norm{\tuple{k}}{\infty} = n\}$.

Our first lemma is a simple observation about the jump rates of the identity coupling of $\jumpKernel_{n \mid \xi}$ and $\jumpKernel_{n \mid \zeta}$ for any pair $\xi, \zeta \in \feasibleCountingMeasures$   
\begin{lemma} \label{lemma:jump_rate_bounds}
    Let $n \in \N$ and $\xi, \zeta \in \feasibleCountingMeasures$, and let $\jumpKernel$ be the identity coupling of $\jumpKernel_{n \mid \xi}$ and $\jumpKernel_{n \mid \zeta}$.
    For every $\tuple{k} \in \vertices_n$ set $\disagreement_{\tuple{k}} \coloneqq \{(\eta_1, \eta_2) \in \countingMeasures_{n \mid \xi} \times \countingMeasures_{n \mid \zeta} \mid (\eta_1 \symdiff \eta_2)_{\subregion_{\tuple{k}}} > 0\}$.
    Then $\jumpKernel$ satisfies the following properties:
    \begin{enumerate}
        \item For all $\tuple{k} \in \innerVertices_n$ and every $(\eta_1, \eta_2) \notin \disagreement_{\tuple{k}} \cup \bigcup_{\tuple{j} \in \neighborhood_n(\tuple{k})} \disagreement_{\tuple{j}}$ it holds that $\jumpKernel((\eta_1, \eta_2), \disagreement_{\tuple{k}}) = 0$.
        \label{lemma:jump_rate_bound:local}
        \item For all $\tuple{k} \in \vertices_n$ and every $(\eta_1, \eta_2) \notin \disagreement_{\tuple{k}}$ it holds that $\jumpKernel((\eta_1, \eta_2), \disagreement_{\tuple{k}}) \le \activity \eulerE^{\localStability} \range^{\dimension}$. \label{lemma:jump_rate_bound:general}
    \end{enumerate}
\end{lemma}

\begin{proof}
    Let $D_1, D_2, D_{\cap}, B_1, B_2, B_{\cap}$ be as in \Cref{lemma:identity_coupling}.
    Set $\subregion = \subregion^{(n)}$ to simplify notation, and observe that for all $\tuple{k} \in \vertices_n$ and $\eta_1, \eta_2 \notin \disagreement_{\tuple{k}}$ it holds that 
    \begin{align*}
        \jumpKernel((\eta_1, \eta_2), \disagreement_{\tuple{k}}) 
        &= B_1((\eta_1, \eta_2), \disagreement_{\tuple{k}}) + B_2((\eta_1, \eta_2), \disagreement_{\tuple{k}}) \\
        &= \activity \int_{\subregion_{\tuple{k}}} \absolute{\eulerE^{-\influence(x, \eta_1 + \xi_{\subregion^c})} - \eulerE^{-\influence(x, \eta_2 + \zeta_{\subregion^c})}} \volume(\diff x) .
    \end{align*}
    Using \Cref{lemma:influence_bound}, the fact that $\potential$ is locally stable and that $\volume(\subregion_{\tuple{k}}) = \range^{\dimension}$ proves (\ref{lemma:jump_rate_bound:general}).

    For (\ref{lemma:jump_rate_bound:local}), note that, if $\tuple{k} \in \innerVertices_n$, it holds that $\dist(x, y) \ge \range$ for any $x \in \subregion_{\tuple{k}}$ and $y \notin \subregion_{\tuple{k}} \cup \bigcup_{\tuple{j} \in \neighborhood_n(\tuple{k})} \subregion_{\tuple{j}}$.
    In particular, since the range of $\potential$ is bounded by $\range$, we have $\potential(x, y) = 0$ for any such pair of points.
    Hence, it is easily checked that $\influence(x, \eta_1 + \xi_{\subregion^c}) = \influence(x, \eta_2 + \zeta_{\subregion^c})$ for every $(\eta_1, \eta_2) \notin \disagreement_{\tuple{k}} \cup \bigcup_{\tuple{j} \in \neighborhood_n(\tuple{k})} \disagreement_{\tuple{j}}$, and consequently $\jumpKernel((\eta_1, \eta_2), \disagreement_{\tuple{k}}) = 0$ as claimed.
\end{proof}

We now combine \Cref{lemma:jump_rate_bounds} with the process construction in \Cref{thm:construction} and \Cref{lemma:separator_set,lemma:ordered_stopping_times} to prove the following lemma.
\begin{lemma} \label{lemma:percolation}
    Let $n \in \N$, let $\xi, \zeta \in \feasibleCountingMeasures$ and let $(P_t), (Q_t)$ be the transition families associated with $\jumpKernel_{n \mid \xi}$ and $\jumpKernel_{n \mid \zeta}$.
    For every $m \in \N_0$ with $m < n$ and all $t \in \R_{\ge 0}$ with $t < \frac{n - m}{\eulerE^2  (6m + 3)^{\dimension} \range^{\dimension} \activity \eulerE^{\localStability}}$ it holds that
    \[
        \projDtv{P_t(\pmb{0}, \cdot)}{Q_t(\pmb{0}, \cdot)}{\subregion^{(m)}} \le \eulerE^{- (n - m)} .
    \]
\end{lemma}

\begin{proof}
    Let $\jumpKernel$ be the identity coupling of $\jumpKernel_{n \mid \xi}$ and $\jumpKernel_{n \mid \zeta}$ as given by \Cref{lemma:identity_coupling}.
    By \Cref{lemma:coupling,lemma:jump_kernel_non_exp}, we know that $\jumpKernel$ is non-explosive, and its associated transition family is a coupling of $(P_t)$ and $(Q_t)$ in the sense of \Cref{def:coupling_markov}. 
    Next, let $(\Omega, \mathcal{F}, \Pr, (X_t))$ be the Markov process associated with the jump kernel $\jumpKernel$ and the starting distribution $\delta_{\pmb{0}} \otimes \delta_{\pmb{0}}$ as given by \Cref{thm:construction}.
    For $\tuple{k} \in \vertices_n$, define $\disagreement_{\tuple{k}}$ as in \Cref{lemma:jump_rate_bounds} and set $T_{\tuple{k}} \coloneqq \inf\{t \in \R_{\ge 0} \mid X_t \in \disagreement_{\tuple{k}}\}$.
    Using \Cref{lemma:coupling_ineq} and the fact that $(X_t)$ is a Markov process with a transition function that is a coupling of $(P_t)$ and $(Q_t)$, and with initial law $\delta_{\pmb{0}} \otimes \delta_{\pmb{0}}$, we obtain
    \[
        \projDtv{P_t(\pmb{0}, \cdot)}{Q_t(\pmb{0}, \cdot)}{\subregion^{(m)}}
        \le \sum_{\tuple{k} \in \vertices_{m}} \Pr[X_t \in \disagreement_{\tuple{k}}] \le \sum_{\tuple{k} \in \vertices_{m}} \Pr[T_{\tuple{k}} \le t].
    \]
    We proceed by deriving a uniform upper bound on $\Pr[T_{\tuple{k}} \le t]$ for every $\tuple{k} \in \vertices_{m}$.
    To this end, we first need to introduce some additional terminology and notation.
    Recall the definition of the graph $\graph_n = (\vertices_n, \edges_n)$. 
    For $\ell \in \N$, we call a sequence of vertices $(\tuple{k}_{1}, \dots, \tuple{k}_{\ell}) \in \vertices_n^{\ell}$ a \emph{path of length $\ell$} in $\graph_n$ if for all $2 \le i \le \ell$ it holds that $\tuple{k}_i \in \neighborhood_n(\tuple{k}_{i-1})$ and for all $1 \le i < j \le \ell$ it holds that $\tuple{k}_i \neq \tuple{k}_j$ (i.e., $(\tuple{k}_{1}, \dots, \tuple{k}_{\ell})$ is a simple path in $G_n$ in the graph theoretic sense).
    We write $\paths_{\ell}$ for the set of all paths of length $\ell$ in $\graph_n$.
    Moreover, for two sets $W_1, W_2 \subseteq \vertices_n$, we denote the set of paths of length $\ell$ from $W_1$ to $W_2$ by
    \[
        \paths_{\ell}(W_1, W_2) \coloneqq \{(\tuple{k}_{1}, \dots, \tuple{k}_{\ell}) \in \paths^{\ell} \mid \tuple{k}_1 \in W_1, \tuple{k}_{\ell} \in W_2\},
    \]
    and we write $\paths(W_1, W_2) \coloneqq \bigcup_{\ell \in \N} \paths_{\ell}(W_1, W_2)$ for all paths from $W_1$ to $W_2$ regardless of their length. 

    We proceed by stating the claim that formalizes the following intuition: Due to the bounded range of the potential, a disagreement between the two Markov processes in some box $\subregion_{\tuple{k}}$ for $\tuple{k} \in \vertices_n$ can only be generated if there already was a disagreement in any box that is adjacent to $\tuple{k}$ in the graph $\graph_n$.
    In particular, if there is a disagreement in some box $\tuple{k} \in \vertices_m$ at time $t$ (i.e., $X_t \in \disagreement_{\tuple{k}}$), then there must be some path $(\tuple{k}_1, \dots, \tuple{k}_{\ell})$ such that $\tuple{k}_{1} = \tuple{k}$ and $\tuple{k}_{\ell} \in \outerVertices_n$ that this disagreement can be traced back to.
    Formally, it must hold that the respective hitting times satisfy $T_{\tuple{k}_{\ell}} < T_{\tuple{k}_{\ell - 1}} < \dots < T_{\tuple{k}_1} \le t$.
    We use \Cref{lemma:separator_set} and part (\ref{lemma:jump_rate_bound:local}) of \Cref{lemma:jump_rate_bounds} to formalize this idea.
    \begin{claim} \label{claim:path_propagation}
        For all $\tuple{k} \in \vertices_m$ and all $t \in \R_{\ge 0}$ it holds that
        \[
            \Pr[T_{\tuple{k}} \le t] \le \sum_{(\tuple{k}_1, \dots, \tuple{k}_{\ell}) \in \paths(\tuple{k}, \outerVertices_n)} \Pr[T_{\tuple{k}_{\ell}} < T_{\tuple{k}_{\ell - 1}} < \dots < T_{\tuple{k}_1} \le t].
        \]
    \end{claim}
    Combining \Cref{claim:path_propagation} with \Cref{lemma:ordered_stopping_times} and part (\ref{lemma:jump_rate_bound:general}) of \Cref{lemma:jump_rate_bounds}, we get
    \[
        \Pr[T_{\tuple{k}} \le t] \le \sum_{\ell \in \N} \size{\paths_{\ell}(\tuple{k}, \outerVertices_n)} \cdot (1 - F_{\rho}(\ell - 1)) ,
    \]
    where $F_{\rho}$ is the cumulative distribution function of a Poisson random variable of rate $\rho = \activity \eulerE^{\dimension} \range^{\dimension} t$.
    Next, note that $\size{\paths_{\ell}(\tuple{k}, \outerVertices_n)} = 0$ for all $\ell \le n - m$ and $\size{\paths_{\ell}(\tuple{k} , \outerVertices_n)} \le 3^{\dimension \cdot \ell}$ for all $\ell > n - m$.
    Further, using standard tail bounds for the Poisson distribution, we get that for all $\ell \ge \eulerE^{c + 1} \rho$ with $c \ge \dimension \ln(6 m + 3) + 1$ it holds that $1 - F_{\rho}(\ell - 1) \le \eulerE^{- c \ell} = \eulerE^{-\ell} 3^{-\dimension \ell}  (2m + 1)^{-\dimension \ell}$.
    Hence, we have
    \[
        \Pr[T_{\tuple{k}} \le t] \le \sum_{\ell > n - m} \eulerE^{- \ell} (2m+1)^{- \dimension \ell} \le \size{\vertices_m}^{-1} \cdot \eulerE^{- (n - m)} 
    \]
    and summing over $\tuple{k} \in \vertices_m$ concludes the proof.
\end{proof}

\begin{proof}[Proof of \Cref{claim:path_propagation}]
    For $\ell \in \N$ define 
    \begin{align*}
        S^{\text{f}}_{\ell} (\tuple{k}) &\coloneqq \{(\tuple{k}_1, \dots, \tuple{k}_j) \in \paths_j \mid j \le \ell, \tuple{k}_1 = \tuple{k}, \tuple{k}_j \in \outerVertices_n, \forall i < j: \tuple{k}_i \in \innerVertices_n \} \\
        S^{\text{o}}_{\ell} (\tuple{k}) &\coloneqq \{(\tuple{k}_1, \dots, \tuple{k}_\ell) \in \paths_{\ell} \mid \tuple{k}_1 = \tuple{k}, \forall i \le \ell: \tuple{k}_i \in \innerVertices_n \} \\
        S_{\ell}(\tuple{k}) &\coloneqq S^{\text{f}}_{\ell} \cup S^{\text{f}}_{\ell}. 
    \end{align*}
    In prose: $S^{\text{f}}_{\ell}(\tuple{k})$ is the set of all paths of length at most $\ell$ that start with $\tuple{k}$, end in $\outerVertices_n$ and only go via $\innerVertices_n$ in-between.
    $S^{\text{o}}_{\ell} (\tuple{k})$ is the set of paths of length exactly $\ell$ that start in $\tuple{k}$ and never visit $\outerVertices$.
    $S_{\ell}(\tuple{k})$ is the union of these disjoint sets.

    We will prove our claim by showing inductively that for all $\ell \in \N$,
    \begin{align} \label{claim:path_propagation:main_equation}
        \Pr[T_{\tuple{k}} \le t] \le \sum_{(\tuple{k}_1, \dots, \tuple{k}_{j}) \in S_{\ell}(\tuple{k})} \Pr[T_{\tuple{k}_{j}} < T_{\tuple{k}_{j - 1}} < \dots < T_{\tuple{k}_1} \le t].
    \end{align}
    To see that this indeed proves the claim, note that $S^{\text{f}}_{\ell}(\tuple{k}) \subseteq \paths(\tuple{k}, \outerVertices_n)$ for all $\ell \in \N$ and that $S^{\text{o}}_{\ell}(\tuple{k}) = \emptyset$ for $\ell > \size{\innerVertices_n}$.

    For the induction base, note that the statement holds trivially for $\ell = 1$ since $S_{1}(\tuple{k}) = \{(\tuple{k})\}$.
    Suppose \eqref{claim:path_propagation:main_equation} holds for some $\ell \in \N$.
    By splitting up the sum in \eqref{claim:path_propagation:main_equation} into a sum for $S^{\text{o}}_{\ell}(\tuple{k})$ and $S^{\text{f}}_{\ell}(\tuple{k})$ and noting that $S^{\text{f}}_{\ell}(\tuple{k}) \subseteq S^{\text{f}}_{\ell+1}(\tuple{k})$, we see that it suffices to show that
    \begin{align*}
        &\sum_{(\tuple{k}_1, \dots, \tuple{k}_{\ell}) \in S^{\text{o}}_{\ell}(\tuple{k})} \Pr[T_{\tuple{k}_{\ell}} < T_{\tuple{k}_{\ell - 1}} < \dots < T_{\tuple{k}_1} \le t] \\
        &\hspace{4em} \le \sum_{(\tuple{k}_1, \dots, \tuple{k}_{\ell + 1}) \in S^{\text{o}}_{\ell + 1}(\tuple{k})} \Pr[T_{\tuple{k}_{\ell + 1}} < T_{\tuple{k}_{\ell}} < \dots < T_{\tuple{k}_1} \le t] \\
        &\hspace{4em}+ \sum_{(\tuple{k}_1, \dots, \tuple{k}_{\ell + 1}) \in S^{\text{f}}_{\ell + 1}(\tuple{k}) \setminus S^{\text{f}}_{\ell}(\tuple{k})} \Pr[T_{\tuple{k}_{\ell + 1}} < T_{\tuple{k}_{\ell}} < \dots < T_{\tuple{k}_1} \le t]  .  
    \end{align*}
    To this end, fix some $(\tuple{k}_1, \dots, \tuple{k}_{\ell}) \in S^{\text{o}}_{\ell}(\tuple{k})$ and note that, by definition, $\tuple{k}_{\ell} \in \innerVertices_n$.
    Set $A \coloneqq \{T_{\tuple{k}_{\ell}} < T_{\tuple{k}_{\ell - 1}} < \dots < T_{\tuple{k}_1} \le t\}$ and note that $A \subseteq \{T_{\tuple{k}_{\ell}} < \infty\}$.
    Further, define $T_{\neighborhood} \coloneqq \inf\{s \in \R_{\ge 0} \mid X_{s} \in \bigcup_{\tuple{j} \in \neighborhood_n(\tuple{k}_{\ell})} \disagreement_{\tuple{j}}\}$.
    By \Cref{lemma:separator_set} and part (\ref{lemma:jump_rate_bound:local}) of \Cref{lemma:jump_rate_bounds}, it holds that
    \[
        \Pr[A] = \Pr[T_{\neighborhood} < T_{\tuple{k}_{\ell}}, A] \le \sum_{\tuple{j} \in \neighborhood_n(\tuple{k}_{\ell})} \Pr[T_{\tuple{j}} < T_{\tuple{k}_{\ell}}, A].
    \]
    Now, observe that, if $\tuple{j} = \tuple{k}_i$ for any $1 \le i \le \ell$, then $\{T_{\tuple{j}} < T_{\tuple{k}_{\ell}}, A\} \subseteq \{T_{\tuple{j}} < T_{\tuple{j}}\}$ and hence $\Pr[T_{\tuple{j}} < T_{\tuple{k}_{\ell}}, A] = 0$.
    Thus, we only need to consider $\tuple{j} \in \neighborhood_n(\tuple{k}_{\ell})$ that extend $(\tuple{k}_1, \dots, \tuple{k}_{\ell})$ to a path of length $\ell + 1$ (i.e., $(\tuple{k}_1, \dots, \tuple{k}_{\ell}, \tuple{j}) \in \paths_{\ell+1}$).
    If $\tuple{j} \in \outerVertices_n$ then $(\tuple{k}_1, \dots, \tuple{k}_{\ell}, \tuple{j}) \in S^{\text{f}}_{\ell + 1}(\tuple{k}) \setminus S^{\text{f}}_{\ell}(\tuple{k})$.
    Otherwise, we have $(\tuple{k}_1, \dots, \tuple{k}_{\ell}, \tuple{j}) \in S^{\text{o}}_{\ell + 1}(\tuple{k})$.
    Since distinct paths from $S^{\text{o}}_{\ell}(\tuple{k})$ also produce distinct extensions, this concludes the induction step and proves the claim.
\end{proof}

\section{Uniqueness of infinite-volume Gibbs measures} \label{sec:uniqueness}
We will now prove our main uniqueness result for infinite-volume Gibbs measures.
\begin{theorem} \label{thm:uniqueness}
    Suppose $(\groundspace, \dist)$ is the $\dimension$-dimensional Euclidean space for some $\dimension \in \N$, equipped with the respective Lebesgue measure.
    Let $\potential$ be a finite-range pair potential that is locally stable with constant $\localStability$ and weakly tempered with constant $\weakTempered_{\potential}$.
    For all $\activity < (\eulerE^{\localStability} \weakTempered_{\potential})^{-1}$ it holds that $\size{\gibbsMeasures(\activity, \potential)} \le 1$.
\end{theorem}

\begin{proof}
    Suppose $\gibbsMeasures(\activity, \potential)$ is non-empty.
    Our goal is to use \Cref{lemma:uniqueness_from_ssm} to argue that, for all $\gibbs_1, \gibbs_2 \in \gibbsMeasures(\activity, \potential)$, it holds that $\gibbs_1 = \gibbs_2$.
    To this end, set $\subregion_k = [-(k + \frac{1}{2}) \range, (k + \frac{1}{2}) \range]^{\dimension}$ where $\range$ is an upper-bound on the range of $\potential$, and note that the sequence $(\subregion_k)_{k \in \N}$ is non-decreasing and satisfies assumption (\ref{lemma:uniqueness_from_ssm:cover}) of \Cref{lemma:uniqueness_from_ssm}.
    Moreover, note that by \Cref{lemma:feasible}, it holds that $\gibbs_1(\feasibleCountingMeasures) = 1$ and $\gibbs_2(\feasibleCountingMeasures) = 1$.
    Hence, it suffices to show that for every $k \in \N$
    \[
        \lim_{n \to \infty} \sup_{\xi, \zeta \in \feasibleCountingMeasures} \projDtv{\gibbs_{\subregion_{k + n} \mid \xi}}{\gibbs_{\subregion_{k + n} \mid \zeta}}{\subregion_k} = 0 .
    \]
    Let $\delta > 0$ be as in \Cref{lemma:mixing} and note that for every $\varepsilon > 0$, we can choose $n \ge -\ln \varepsilon$ large enough, depending on $\varepsilon$, $\activity$, $\localStability$, $\range$, $\dimension$, $\weakTempered_{\potential}$ and $k$, such that there is some $t \in \R_{\ge 0}$ with
    \begin{align}
        \delta^{-1} \ln\left( \frac{ \activity (2n + 2k + 1)^{\dimension} \range^{\dimension} \eulerE^{\localStability}}{\varepsilon} \right) \le t < \frac{n}{\eulerE^{2} (6k + 3)^{\dimension} \range^{\dimension} \activity \eulerE^{\localStability}} . \label{eq:choose_t}
    \end{align}
    Fix any such $n$ and let $\xi, \zeta \in \feasibleCountingMeasures$. 
    Let $(P_t)$ and $(Q_t)$ be the transition families associated with the jump kernels $\jumpKernel_{\subregion_{k + n} \mid \xi}$ and $\jumpKernel_{\subregion_{k + n} \mid \zeta}$ as in \eqref{eq:jump_kernel}.
    By the triangle inequality, we have
    \[
        \projDtv{\gibbs_{\subregion_{k + n} \mid \xi}}{\gibbs_{\subregion_{k + n} \mid \zeta}}{\subregion_k}
        \le \projDtv{\gibbs_{\subregion_{k + n} \mid \xi}}{P_t(\pmb{0}, \cdot )}{\subregion_k} + \projDtv{P_t(\pmb{0}, \cdot )}{Q_t(\pmb{0}, \cdot )}{\subregion_k} + \projDtv{Q_t(\pmb{0}, \cdot )}{\gibbs_{\subregion_{k + n} \mid \zeta}}{\subregion_k}
    \]
    for every $t \in \R_{\ge 0}$.
    Choosing any $t$ as in \eqref{eq:choose_t}, \Cref{lemma:mixing} yields
    \[
        \projDtv{\gibbs_{\subregion_{k + n} \mid \xi}}{P_t(\pmb{0}, \cdot )}{\subregion_k} \le \dtv{\gibbs_{\subregion_{k + n} \mid \xi}}{P_t(\pmb{0}, \cdot )} \le \varepsilon
    \]
    and analogously
    \[
        \projDtv{Q_t(\pmb{0}, \cdot )}{\gibbs_{\subregion_{k + n} \mid \zeta}}{\subregion_k} \le \dtv{Q_t(\pmb{0}, \cdot )}{\gibbs_{\subregion_{k + n} \mid \zeta}} \le \varepsilon.
    \]
    Further, for the same $t$, \Cref{lemma:percolation} and $n \ge - \ln \varepsilon$ yield
    \[
         \projDtv{P_t(\pmb{0}, \cdot )}{Q_t(\pmb{0}, \cdot )}{\subregion_k} \le \eulerE^{- n} \le \varepsilon.
    \]
    Since our choice of $n$ does not depend on $\xi$ and $\zeta$, we get
    \[
        \sup_{\xi, \zeta \in \feasibleCountingMeasures} \projDtv{\gibbs_{\subregion_{k + n} \mid \xi}}{\gibbs_{\subregion_{k + n} \mid \zeta}}{\subregion_k} \le 3 \varepsilon
    \]
    for every $n$ large enough, which proves the claim.
\end{proof}

Combining \Cref{thm:uniqueness} with known existence results for infinite-volume Gibbs measures, such as \Cref{thm:existence}, we obtain the following corollary.
\begin{corollary} \label{cor:uniqueness}
    In the setting of \Cref{thm:uniqueness}, it holds that $\size{\gibbsMeasures(\activity, \potential)} = 1$ for all $\activity < (\eulerE^{\localStability} \weakTempered_{\potential})^{-1}$.
\end{corollary}

\section{Discussion} \label{sec:discussion}
We conclude the paper by comparing our result to the existing literature. 
We further give a detailed discussion of our assumptions and potential ways to extend our results in future work. 

\subsection{Comparison to existing bounds for absence of phase transitions}
To compare our results to the existing literature, it is useful to introduce the inverse temperature parameter $\beta \ge 0$.
For any fixed pair potential $\potential$ on $\R^{\dimension}$,
we denote by $\potential_{\beta} \coloneqq \beta \potential$ the pair potential at inverse temperature $\beta$.
Intuitively, changing $\beta$ scales the strength of interactions between particles: in the small temperature limit ($\beta \to \infty$) interactions between particles become more relevant for the behavior of the system, whereas in the large temperature limit ($\beta \to 0$) only hard-core interactions (cases where $\potential(\cdot, \cdot) = \infty$) prevail. 
We write $\weakTempered_{\beta} \coloneqq \weakTempered_{\potential_{\beta}}$ for the associated weak temperedness constant and $\tempered_{\beta} \coloneqq \sup_{x \in \R^{\dimension}} \int_{\R^{d}} |1 - \eulerE^{- \beta \potential(x, y)}| \diff y$ for the associated temperedness constant. 
Note that, if $\potential$ is (weakly) tempered, the same applies to $\potential_{\beta}$ for all $\beta$ (see for example \cite[Exercise 5.3]{jansen2018gibbsian}).
Moreover, if $\potential$ is locally stable with constant $\localStability$, then $\potential_{\beta}$ is locally stable with constant $\beta \localStability$.
Our bound for uniqueness of the infinite-volume Gibbs measure then translates into $\activity < (\eulerE^{\beta \localStability} \weakTempered_{\beta})^{-1}$ with the additional assumption that $\potential$ has bounded range.

We start by comparing our result with the classical Penrose--Ruelle bound of $\activity < (\eulerE^{2 \beta \stability +1} \tempered_{\beta})^{-1}$ for stable pair potentials, where $\stability$ is the stability constant of $\potential$, which first appeared in \cite{penrose1963convergence,ruelle1963correlation}.
Within this regime of $\activity$, uniqueness of infinite-volume Gibbs measures can be established by considering the Kirkwood--Salsburg equations, which describe the point density functions of any infinite-volume Gibbs point process as a fix point of a linear operator.
Arguing that this linear operator contracts under a suitable metric proves uniqueness of that fixed point and by extension shows that there is at most one Gibbs measure (see \cite[Chapters 5.6-5.8]{jansen2018gibbsian} for a detailed discussion).
In the setting of locally stable potentials, this approach yields uniqueness of infinite-volume Gibbs measures for $\activity < (\eulerE^{\beta \localStability  +1} \tempered_{\beta})^{-1}$.
Hence, our result yields an improvement by a factor of $\eulerE \tempered_{\beta}/\weakTempered_{\beta}$ at the cost of relying on a bounded-range assumption.
In particular, this improvement is lower-bounded by a factor of $\eulerE$, and if $\potential$ has some non-trivial negative part (i.e., attractive interactions), it becomes more substantial for large $\beta$ (i.e., low temperature).

The second bound we want to compare our result to is a recently obtained result by Qidong He \cite{he2024analyticity}.
They showed that for locally stable, tempered, and translation and rotation invariant pair potentials, the (empty-boundary) infinite-volume pressure $\lim_{n \to \infty}\volume(\subregion_n)^{-1} \log(\partitionFunction_{\subregion \mid \pmb{0}}(\activity))$ is an analytic function\footnote{Here, the limit is taken along a suitable sequence $\subregion_n \nearrow \R^{\dimension}$. See \cite{jansen2018gibbsian} for a detailed discussion.} for any $\activity < \eulerE^{2 (1 - W(\eulerE A_{\beta}/\tempered_{\beta}))}(\eulerE^{\beta \localStability + 1} \tempered_{\beta})^{-1}$, where $W$ denotes Lambert's W-function and $A_{\beta} \coloneqq \int_{\R^{\dimension}} \ind{\potential(0, x) < 0} \cdot (\eulerE^{-\beta \potential(0, x)} - 1) \diff x$.
The result is obtained by adapting a technique introduced by Michelen and Perkins \cite{michelen2020analyticity} for repulsive potentials (i.e., $\potential \ge  0$), which relies on showing contraction of a recursive integral identity for the (generalized complex) point density.
We point out that this result aims at a different notion of absence of a phase transition, and, to our knowledge, does not a priori imply uniqueness of the infinite-volume Gibbs measure in the DLR sense.
However, we believe that their proof can be modified to show uniqueness up to the same activity threshold via similar arguments as in \cite[Section 4]{michelen2020analyticity}, while handling the effect of the boundary condition as part of the modulation of the activity function in the sense of \cite{he2024analyticity}.
With this in mind, we observe that the relation of their bound to ours heavily depends on the inverse temperature $\beta$ and the potential in question.
In particular, if the potential is purely repulsive or the temperature is high (i.e., small $\beta$), their bound is up to a factor of $\eulerE$ larger than ours.
On the other hand, for potentials with a non-trivial attractive parts and low temperature (i.e., large $\beta$) our bound performs better as $W(\eulerE A_{\beta}/\tempered_{\beta}) \in [0, 1)$ while $\tempered_{\beta}/\weakTempered_{\beta} \to \infty$ for $\beta \to \infty$.
Again, this improvement comes at the cost that we require the potential to have bounded range.

Lastly, we compare our result to the bound for analyticity of the infinite-volume pressure (again with empty boundary) by Procacci and Yuhjtman \cite{procacci2017convergence}.
They show that for stable and weakly tempered potentials, the infinite-volume pressure is analytic for all $\activity < (\eulerE^{\beta \stability + 1} \weakTempered_{\beta})^{-1}$ by proving absolute convergence of the cluster expansion, a series expansion of $\log(\partitionFunction_{\subregion \mid \pmb{0}}(\activity))$.
In the setting of locally stable potentials, this translates to a bound of $\activity < (\eulerE^{\beta\localStability/2 + 1} \weakTempered_{\beta})^{-1}$.
Again, this approach aims at a different notion of phase transition.
However, as opposed to the result we discussed before, it is not known whether uniqueness of the infinite-volume Gibbs measure can be proven for the same activity regime.
On the one hand, the precise relation between conditions for convergence of the cluster expansion and uniqueness of Gibbs measures is elusive (see \cite{jansen2019cluster} for a more comprehensive discussion). 
On the other hand, it is not clear whether the activity bound is preserved when accounting for the effect of boundary conditions on the convergence criterion (see \cite{fialho2021cluster} for a similar observation).
If uniqueness of the infinite-volume Gibbs measure were shown to hold up to the convergence bound from \cite{procacci2017convergence}, then this would give an improvement over our result in the high-temperature regime ($\beta < 2 / \localStability$), while our result would still be stronger at low temperature.
However, a proof of this fact would be non-trivial, and it is not clear whether it should hold at all.
It would make for an interesting contrast to our previous comparison to \cite{he2024analyticity}, where our bound was better at low-temperature.

We finish our discussion of related uniqueness results by noting that for the special case of purely non-negative potentials, stronger results are known due to Michelen and Perkins \cite{michelen2020analyticity,michelen2020connective}, which improve our uniqueness regime by a factor of at least $\eulerE$. 
It remains open whether such an improvement carries over to potentials with non-trivial negative part.
As discussed before, a promising step towards this was presented in \cite{he2024analyticity}, which maintains this improvement at small $\beta$.

\subsection{Assumptions and extensions}
The first assumption we discuss is that $\potential$ needs to have bounded range.
This is in fact only required in the disagreement percolation argument in \Cref{sec:percolation}, where we use it to make sure that disagreements between any pair of birth-death processes only spread locally.
In contrast, our second main ingredient, which is that each birth-death process rapidly converges to a finite-volume Gibbs distribution, does not rely on it.
A similar observation is true for the fact that we restrict our main result to Euclidean space.
While the mixing result in \Cref{lemma:mixing} holds in any complete separable metric space, the disagreement percolation in \Cref{lemma:percolation} requires additional information on the growth rate of the underlying space.
Both of these assumptions are inherited from the discrete origin of our proof technique \cite{dyer2004mixing}, which is in studying lattice gasses that naturally satisfy such conditions.
However, one could hope to remove the bounded-range assumption by applying a suitable truncation of the pair potential at some range (maybe requiring additional assumptions on the decay rate of the pair potential). 

The second main assumption we discuss is local stability.
This assumption is relevant throughout most of our proofs: we use it for proving non-explosiveness of the jump kernel (in \Cref{lemma:jump_kernel_non_exp}), to bound the rate at which the Markov transition family converges to its stationary distribution (in \Cref{lemma:mixing} and \Cref{claim:drift}) and to bound the overall rate at which two coupled processes spawn a disagreement (in \Cref{lemma:percolation} and \Cref{lemma:jump_rate_bounds}).
All of these applications make use of the fact that local stability permits a bounding of the rate at which new points are generated uniformly over the configurations space of the process.
Interestingly, other approaches to prove uniqueness, such as proving contraction of the Kirkwood--Salsburg operator, encounter similar difficulties; there the problem is solved by observing that stability implies that at least one point in any configuration interacts in a ``locally stable manner'' with the rest of the configuration (see \cite[Chapter 4.2]{ruelle1999statistical} for details).
However, it is not clear how a similar strategy can be implemented here.

Finally, we would like to take a more detailed look at our obtained activity bound of $\activity < (\eulerE^{\localStability} \weakTempered_{\potential})^{-1}$.
This bound is solely due to the activity regime for which we can prove that each birth-death process converges rapidly to its stationary distribution.
Therefore any improvement on the activity bound for rapid mixing would immediately improve the uniqueness regime.
The main candidate for approaching this is to use a more suitable metric $f$ on the configuration space in the proof of \Cref{lemma:mixing} (see also \Cref{lemma:jump_process_mixing}) to show contraction of the coupling.
For the hard-sphere model (i.e., $\potential(x, y) = \infty \cdot \ind{\dist(x, y) < r}$ for some fixed $r \in \R_{\ge 0}$), this approach was used by Helmuth, Perkins and Petti \cite{helmuth2022correlation} to gain an additional factor of $2$ on the activity bound.
The intuition behind their metric was to weight pairs of configurations based on how likely they are to spawn disagreements.
In particular, this approach exploits the fact that, in the hard-sphere model, adding a point only decreases this likelihood. 
While this idea generalizes nicely for other repulsive potentials (i.e., $\potential \ge 0$), it becomes much harder to apply with attractive interactions, where adding a point might make it more likely to spawn new disagreements.
However, focusing on more specific models could be a reasonable step towards finding better conditions for rapid convergence of the associated birth-death dynamics.

\bibliographystyle{abbrv}
\bibliography{references}

\end{document}